\documentclass{article}
\usepackage{geometry}
\usepackage{graphicx} 
\usepackage{amsfonts}
\usepackage{amsmath}
\usepackage{amssymb}
\usepackage[bookmarksnumbered=true, hidelinks]{hyperref}
\usepackage[T1]{fontenc}
\usepackage[utf8]{inputenc}
\usepackage{amsthm}
\usepackage{enumerate}
\usepackage[shortlabels]{enumitem}
\usepackage{authblk}
\usepackage{cite}
\usepackage{comment}

\usepackage{microtype}

\theoremstyle{plain}
\newtheorem{thrm}{Theorem}[section]
\newtheorem{lmm}[thrm]{Lemma}
\newtheorem{crllr}[thrm]{Corollary}
\newtheorem{prpstn}[thrm]{Proposition}
\theoremstyle{definition}
\newtheorem{dfntn}[thrm]{Definition}
\newtheorem{nt}[thrm]{Note}
\newtheorem{cnjctr}[thrm]{Conjecture}
\newtheorem{xmpl}[thrm]{Example}

\def\j{\mathbf{j}}
\def\kk{\mathbf{k}}
\def\ll{\mathbf{l}}

\title{  \bfseries{ \textsc{ Weights of strongly nilpotent special multi-flags}}}
\author{ {\large \scshape Bartłomiej Sikorski}\\
{\small Faculty of Physics, University of Warsaw, Pasteura 5, 02-093 Warsaw, Poland}}
\date{}

\begin{document}

\maketitle

\begin{abstract}
   This paper is devoted to two distinguished families of distributions: special multi-flags and their lower-rank counterparts, Goursat distributions. It is known that these two families are weakly nilpotent but strongly nilpotent only at special points. Local classification of these objects is still open and only recently singularity classes of special multi-flags were constructed by Mormul. In this paper, we study special multi-flags in the homogeneous case, which is also strongly nilpotent. The main contribution of the present work is to show that the strongly nilpotent germs of multiflags from different singularity classes are pairwise inequivalent as the singularity type uniquely determines the weights of coordinates. In the Goursat case, explicit formulas for weights of tangential geometric classes are presented.  Finally, we discuss the property of strong nilpotency among special multi-flags and state open conjectures.
\end{abstract}

\section{Introduction}

Nonholonomic systems occupy a special place at the meeting point of geometric control theory, sub-Riemannian geometry, and the local theory of differential systems. A configuration of such a system is described by a point $x$ on a smooth manifold $M$ and the system is governed by a distribution $D \subset TM$, whose sections describe the admissible infinitesimal motions, $\dot x(t) \in D_{x(t)}$.
Directions outside $D$, when they can be recovered, appear through Lie brackets; thus accessibility is reflected in the small Lie flag 
generated from $D$\footnote{
See \eqref{def_small_flag} for the definition of a small Lie flag.
}. In the totally nonholonomic, or bracket-generating, case at each point $x\in M$ there exists a finite degree of nonholonomy $d_{NH}(x)$ such that $D_{d_{NH}(x)} = TM$ around $x$.

This bracket-generating mechanism is one of the basic links between nonlinear control theory and geometry. In control theory it is tied to accessibility and local controllability through Lie-bracket criteria \cite{Sussmann1983,Sussmann1987}. In motion planning, the number and arrangement of brackets needed to recover the missing directions influence the local complexity of steering; this is one reason why nilpotent bases and nilpotent approximations play such a central role \cite{Hermes1984NilpotentBF,Lafferriere1993}. In sub-Riemannian geometry the same structure appears through privileged coordinates, weights, and tangent nilpotent models \cite{Bellaiche96,Montgomery2002}.

A particularly transparent example is the classical system of a car pulling trailers. Jean showed that its singular configurations are detected by changes in the degree of nonholonomy and in the associated Lie-bracket growth \cite{Jean_1996}. Thus what appears mechanically as a degeneracy of configurations becomes, geometrically, a stratification problem for rank-two distributions.
The geometric objects abstracted from that example are the Goursat distributions. 
 They are characterized by the slowest possible regular growth of the derived, or big, flag, which is $2,3,4,\ldots,\dim M$ at each point.
They form the first and most thoroughly studied family in a larger hierarchy. Their local theory revealed that even such apparently rigid objects possess a rich supply of singularities. Kumpera-Ruiz normal forms \cite{KumperasurlequivalencedePfaff}, Jean’s stratification coming from the $n$-trailer system \cite{Jean_1996}, and the later geometric interpretation of Montgomery and Zhitomirskii \cite{Montgomery2001GeometricAT} showed that \emph{Goursat germs}
naturally split into classes, each carrying its own small growth vector and its own nonholonomic behaviour.

Here and throughout the paper, the word \emph{germ} emphasizes that the
classification problem is local: distributions are compared in
neighborhoods of distinguished points, up to local diffeomorphisms. This is
the standard viewpoint in the local theory of singular distributions, as in \cite{Zhitomirskii1991NormalFO,Zhitomirskii1995SingularitiesAN}.

Goursat distributions, however, are only the width-one case. The natural higher-width analogues are the special multi-flags. For a fixed integer $m\geq 1$, a special $m$-flag is generated by a rank-$(m+ 1)$ distribution $D$ whose derived flag grows in ranks
$m+1,2m+1,3m+1,\ldots.$
Together with the characteristic and involutivity conditions in the definition, this controlled growth makes the class rigid enough to admit explicit normal forms, yet flexible enough to contain nontrivial singularities.
These objects arise naturally from generalized Cartan prolongations and are closely related to differential systems \cite{Kumpera2002MultiflagSA,SHIBUYA2009793,Mormul2004MultidimensionalCP}.

A decisive feature of both Goursat distributions and special multi-flags is weak nilpotency: locally, one can choose generators whose real Lie algebra is nilpotent. This property is important in control theory because nilpotent systems are the model systems for local motion planning and for first-order nonholonomic approximation. Yet weak nilpotency does not say that the original distribution is already equivalent to its nilpotent approximation. The latter, stronger property is strong nilpotency. Strongly nilpotent germs are, in this sense, the points where the local geometry is exactly captured by its nilpotent model rather than merely approximated by it.

The present paper studies strongly nilpotent special multi-flags through their weights. Given a completely nonholonomic distribution, the small Lie flag determines adapted coordinates and assigns weights to the coordinate directions: a coordinate has weight $s$ if its direction first appears at the $s$-th step of the small flag. These weights are the bookkeeping device behind the nilpotent approximation. They also determine the small growth vector, and hence provide diffeomorphism invariants of the germ. The guiding question is therefore simple to state: how much of the singularity class of a strongly nilpotent special multi-flag is remembered by its sequence of weights?

The main conclusion of this paper is that, on the homogeneous (hence also strongly nilpotent) locus, the singularity code is visible in the weight structure.
More precisely, fix a width $m$ and a length $r$. We prove that two
homogeneous EKR normal forms of special $m$-flags with different lujr
codes have different sequences of weights. These weights determine the
small growth vector in the strongly nilpotent normal form; hence the
corresponding germs have different small growth vectors and are not
locally diffeomorphic. Thus, distinct singularity classes are pairwise
inequivalent within the homogeneous strongly nilpotent locus.

This result continues the local classification programme for multi-flags
developed in \cite{Mormul2004MultidimensionalCP, Mormul_Pelleter_2flags,
Mormul_2023_multi}. The recent construction of singularity classes for
special multi-flags gives geometric meaning to the EKR codes. Here we show
that, after restricting to homogeneous EKR normal forms, these codes are
separated by the associated weight sequences. The proof is based on
recurrence relations for the weights of coordinates read directly from the
EKR construction. In the Goursat case, that is, for $1$-flags, we also
derive explicit formulas for the weights of tangential classes, building
on Jean's analysis of the $n$-trailer system \cite{Jean_1996} and
Mormul's arithmetic description of small growth vectors \cite{SGV_Mor_2004}.

The paper is organized as follows. Section \ref{sec:2} reviews preliminary concepts in the theory of nonholonomic distributions. In Section \ref{sec:3}, we describe normal forms and geometric classes of special multi-flags. In Section \ref{sec:4} we prepare the grounds to present Theorem \ref{main_theorem}, the main result of this work - the inequivalence of strongly nilpotent normal forms, and hence also their geometric classes. Section \ref{sec:5} presents new explicit formulas for weights of coordinates of tangential (strongly nilpotent) Goursat flags. We close with Section \ref{sec:6}, outlining open problems, in particular, if all strongly nilpotent multi-flags are of the form we consider.  Appendix \ref{app:combi} answers a combinatorial question about the number of singularity classes of $m$-flags. Appendix \ref{app:proof} is devoted to the proof of Theorem \ref{main_theorem_rephrased} necessary to show our main result. Appendix \ref{app:recursion} contains proofs of Theorem \ref{An_compact_form_of_thm2} and Theorem \ref{thm:sums_of_Goursats}.

\section{Preliminaries}\label{sec:2}

\subsection{Lie flags}
Distributions that we study in this paper will be characterized by their following flags and their dimensions.
\begin{dfntn}
    For a given totally nonholonomic (bracket generating) distribution $D$ on a manifold $M$, we introduce its \emph{nonholonomy degree}\footnote{Also called \emph{nilpotency order} \cite{Lafferriere1993}.} $d_{NH}$ and its \emph{small Lie flag} $(D_1,D_2,\hdots,D_{d_{NH}})$\footnote{Here, and throughout this paper equalities of distributions will be understood in the sense of germs of distributions at a distinguished point.}
    \begin{align}
    D_1:&=D,\notag\\
    D_{j+1}:&=D_j+[D,D_j] \quad \text{for }j=1,2,\hdots,d_{NH}-1,\label{def_small_flag}\\
    D_{d_{NH}}:&=TM.\notag
\end{align}
\end{dfntn}
\begin{dfntn}
To every point $q\in M$, there is attached a \emph{small growth vector}(sgrv) of the distribution $D$\footnote{In general, the small growth vector is not constant on $M$, but for the local classification of distributions and their singularities we shall consider germs at a point $q$, usually taken to be $0\in \mathbb{R}^n$.} - the sequence of dimensions defined by the growth of the small Lie flag at a point 
\begin{align}
    \text{sgrv}_q =[\dim D_1(q),\hdots,\dim D_{d_{NH}-1}(q),\dim D_{d_{NH}}(q)].
\end{align}
\end{dfntn}
The small growth vector encodes the nonholonomy of the distribution, that is how fast it reaches the full tangent bundle.
\begin{dfntn}
Given a totally nonholonomic distribution $D$ we can also define its \emph{big Lie flag} around a point $q\in M$
\begin{align}\label{eqn:bigflagdef}
    D^r:&=D,\notag\\
    D^{j-1}:&=D^j+[D^j,D^j] \quad \text{for }j=r,r-1,\hdots,1,\\
    D^{0}:&=TM.\notag
\end{align}
\end{dfntn}
Clearly, the elements of the big flag grow at least as fast as elements of the small flag, that is $ D_j\subset D^{r+1-j}$ due to Jacobi identity.
\begin{dfntn}
For every $q\in M$, we also define \emph{big growth vector} of a distribution as a sequence of dimensions of the big Lie flag
\begin{align}\label{def:biggrowvector}
    \text{bgrv}_q=[\dim D^r(q), \hdots, \dim  D^{0}(q)].
\end{align}
\end{dfntn}
As we shall see in the next sections, the equality of big growth vectors does not imply the equality of small growth vectors. Hence, the small growth vector provides much finer stratification of distributions.

\subsection{Weights}
Given a distribution with a sgrv: $ \text{sgrv}_q=(n_1,n_1+n_2,\hdots,n_1+n_2+\hdots+n_{d_{NH}})$\footnote{Note that $n_1=d$.}, one can decompose the tangent space at a point $q\in M$ as 
\begin{align}
    T_qM= \mathbb{R}^{n_1}\oplus\hdots \oplus \mathbb{R}^{n_{d_{NH}}}.
\end{align}

Any point $q\in M$ has a neighborhood $O_q\ni q$ for which there exist an \emph{adapted coordinate map} (\cite{AgrachevMarigo03}, Theorem 3.1), a coordinate chart  $\chi:O_q\rightarrow \mathbb{R}^n$ such that \begin{align}\label{eq:adaptedcoordinates}    \chi(q)=0,\ \quad  \chi_*(D_j(q))=\mathbb{R}^{n_1}\oplus\hdots \oplus \mathbb{R}^{n_j},\ j=1,2,\hdots,d_{NH}.\end{align}
Here $n_j=\dim(D_j|_q/D_{j-1}|_q)$, so $\dim D_j|_q=n_1+\cdots+n_j$.

One can write any $x\in T_qM$ in terms of \emph{adapted coordinates} \cite{AgrachevMarigo03} as
\begin{align}\label{adapted_coordinate}
    x=(x_1,\hdots x_{d_{NH}}) \quad x_i =(x_{i1},\hdots,x_{in_i})\in\mathbb{R}^{n_i} ,i=1,\hdots d_{NH}.
\end{align}
We will use coordinates adapted to the small Lie flag of the distribution $D$.
\begin{dfntn}
    For a completely nonholonomic distribution $D$ on $M$, coordinates \eqref{adapted_coordinate} around
$q\in M$ are said to be \textit{linearly adapted} at $q$ if 
\begin{align}
D_j|_q=(\partial_1,\hdots,\partial_{n_1+\hdots+n_j})\qquad (j=1,\dots,d_{NH}).
\end{align} Here and in what follows, we omit the writing “span” before a set of vector field generators. 
\end{dfntn}

Given linearly adapted coordinates at $q$, we assign weights to coordinate functions by declaring
\begin{align}\label{weight_def}
    w(x_{ij})=i \quad\text{if } \quad  \frac{\partial}{\partial x_{ij}}\in D_i|_q\setminus D_{i-1}|_q,
\end{align}
so that the first $n_1=d$ coordinates have weight $1$, the next $n_2$ coordinates have weight $2$,
and so on up to $d_{NH}$.
Equivalently, the weight of a coordinate direction is the smallest step of the \emph{small} Lie flag
in which that direction appears at $q$.

We extend $w$ to monomial differential operators by
\begin{align}
    w\!\left(x^\beta \frac{\partial^{|\alpha|}}{\partial x^\alpha}\right)
=\sum_{i=1}^{d_{NH}} i\big(|\beta_i|-|\alpha_i|\big),
\end{align}
where $\alpha,\beta$ are multi-indices grouped according to the splitting
$T_qM=\mathbb{R}^{n_1}\oplus\cdots\oplus\mathbb{R}^{n_{d_{NH}}}$.
For a vector field $|\alpha|=1$, all components of non-zero vector fields have weights $w_{\mathrm{VF}}\geq -d_{NH}$. Vector fields belonging to the distribution at $q$ act only at the coordinates of weight $1$, so $w_{D}\geq -1$.

The weight of a coordinate or differential operator is related to the action of one parameter group of dilation associated with $D$ by $D\ni X\mapsto \delta_\lambda(X) = \lambda^{-1}X$ or in terms of coordinate functions
\begin{align}
    \delta_\lambda (x_1,x_2,\hdots,x_{d_{NH}})=(\lambda x_1,\lambda^2 x_2,\hdots,\lambda^{d_{NH}}x_{d_{NH}}).
\end{align}
For a polynomial differential operator, dilations act by
\begin{align}
     \delta_\lambda \left (x^\beta\frac{\partial^{|\alpha|}}{\partial x^\alpha}\right)=\lambda ^{\sum_{i=1}^{d_{NH}}i(|\beta_i|-|\alpha_i|)}x^\beta\frac{\partial^{|\alpha|}}{\partial x^\alpha}=\lambda ^{w}x^\beta\frac{\partial^{|\alpha|}}{\partial x^\alpha}.
\end{align}
We denote weights of linearly adapted coordinates $(x_1,\hdots,x_{n_1},\hdots,x_{n_1+\hdots+n_{d_{NH}}})$, by $w_k:=w(x_k)$ such that
\begin{align}
   w_1=\hdots=w_{n_1}=&1, \notag\\
   w_{n_1+1}=\hdots=w_{n_1+n_2}=&2,\\
   w_{n_1+\hdots+n_j+1}=\hdots=w_{n_1+\hdots+n_j+n_{j+1}}=&j+1\notag .
\end{align}
We will apply weight to vector fields of the distribution $D$. We see that the weight is additive with respect to the multiplication of vector fields, therefore, it will be superadditive with respect to the Lie bracket.
\begin{prpstn}
    Suppose vector fields $X$ and $Y$ consist of homogeneous terms with weights $\geq k$ and $\geq l$, respectively, at $q$. Then $[X, Y]$ consists of terms of the weights $\geq k + l$ at $q$.
\end{prpstn}

\subsection{Nilpotent distributions}

Following \cite{Bellaiche96} and \cite{Mormul2005KRalgebrasoptimal}, we will decompose vector fields in terms of weights \eqref{weight_def}. For that, we shall use adapted coordinates \eqref{adapted_coordinate}. One can see that a vector field $X\in \mathrm{Sec}(D)$ must contain in its expansion monomials of the form $\partial_i,$ with $i\leq d$ which are of the weight $-1$. Another condition is imposed on monomials of the form $a^j\partial_j$, with $j>d$, due to the nonholonomic nature of the distribution. The precise statement of this fact is given below. 
\begin{prpstn}
 Every smooth vector field $X\in D$ has only terms of weights not smaller than
$-1$ in its Taylor expansion in arbitrary coordinates adapted for the germ of $D$. These terms can
be grouped in homogeneous summands
    \begin{align}\label{homogeneous_summands}
    X = X^{(-1)}+X^{(0)}+X^{(1)}+\hdots,
\end{align}
where $X^{(s)}$ is the homogeneous vector field of weight $s$.  
\end{prpstn}
We will denote by $\widehat{X}$ the lowest homogeneous component\footnote{Also called nilpotent, property required by the weight $-1$ as we will see in the Proposition \ref{prop:nilpotentappnilp}}, so $\widehat{X}:=X^{(-1)}$. We define analogous elements for the entire distribution.
\begin{dfntn}\label{def_nilpotent_approximation}
    Given a distribution $D\overset{loc.}=(X_1,\hdots,X_d)$ \footnote{As before, $(\hdots)$ denotes $\text{span}(\hdots)$.}, we define new distribution $\widehat{D}=(\widehat{X_1},\hdots,\widehat{X_d})$ which is called the \emph{nilpotent approximation} of $D$  at $q$.
\end{dfntn}
 For the proof that $\widehat{D}$ is well-defined independently of the choice of adapted coordinates, we refer to \cite{Bellaiche96}.
The most important property of the nilpotent approximation is that it preserves the nonholonomic structure of the distribution encoded in the small growth vector \cite{Bellaiche96}.
\begin{prpstn}\label{prop:nilpotentappnilp}
    The nilpotent approximation $\widehat{D}$ of $D$ has the same small growth vector at $q$ as $D$ (and hence the same degree of nonholonomy $d_{NH}$) and the real Lie algebra generated by $\widehat{X_1},\hdots,\widehat{X_d}$ is a nilpotent Lie algebra of the nilpotency order\footnote{The nilpotency order of a Lie algebra is the minimal number of multiplications in it yielding always
zero.} $d_{NH}$.
\end{prpstn}

Now we introduce a central property of strong nilpotency (introduced in \cite{Agrachev2001ONTS}) that we exploit in this paper.
\begin{dfntn}
    A distribution $D$ is said to be \emph{strongly nilpotent} at a point $q$ if its germ at $q$ is equivalent to its nilpotent approximation $\widehat{D}$ at $q$.
\end{dfntn}
Nilpotent approximations and strongly nilpotent distributions play an important role in sub-Riemannian geometry where those appeared for the first time \cite{Bellaiche96, Agrachev2001ONTS}.

Strong nilpotency allows one to study the nilpotent approximation $\widehat{D}$ instead of the distribution $D$ itself. In such cases, the nilpotent approximation is much more than merely an approximation in the same class of small growth vectors. Operating on nilpotent approximations is convenient as those are independent of the choice of adapted coordinates \cite{Bellaiche96} and one can work with a more transparent model of distribution with a given small growth vector.

There is also a weaker notion of nilpotency.
\begin{dfntn}\label{def:weakly_nilpotent}
    A distribution $D$ is said to be (locally) \emph{weakly nilpotent} at a point $q$ if $D$ possesses a basis
that generates a nilpotent real Lie algebra of vector fields locally near $q$.
\end{dfntn}

The relation between strong and weak nilpotency was first discussed in \cite{Hermes1984NilpotentBF}.
Those two notions are not equivalent. It is worth noting that by Proposition \ref{prop:nilpotentappnilp}, generators of $\hat{D}$ generate the nilpotent Lie algebra, therefore the strong nilpotency implies the weak nilpotency. The converse does not hold.
 There are counterexamples among, seemingly regular, Goursat distributions, given in \cite[Theorem 3]{MomrulGoursatdim7}

\section{Multi-flags and their normal forms}\label{sec:3}

Special $m$-flags were introduced by Kumpera and Rubin in \cite{Kumpera2002MultiflagSA}, and more appropriately defined in \cite{SHIBUYA2009793, Adachi_Jiro}. 
Singularities of special multi-flags were studied by different means including methods of algebraic geometry based on a generalization of deep curve-singularities \cite{CASTRO12, CASTRO17_2, CASTRO17_1} resulting in a stratification different than the one we consider, which is due to Mormul \cite{Mormul_2009Singularity2flags}. We develop approaches focused on strong nilpotency properties of special $m$-flags \cite{Mormul_Pelleter_2flags}.
We shall follow definitions and notation of \cite{Mormul2004MultidimensionalCP, Mormul_2009Singularity2flags, Mormul_2023_multi}, to which we refer for more details.
First, we recall the definition of special multi-flags.
\begin{dfntn}\label{Multi_flag_definition}
    \emph{Special $m$-flag} is a big flag of a distribution such that its big growth vector takes the form
    \begin{align}
        \text{bgrv}_q=(\dim D^r, \hdots, \dim  D^{0})=(m+1,2m+1,\hdots, rm+1,(r+1)m+1),
    \end{align}
    and $D^1$ possesses a corank 1 involutive subdistribution $F$ ($[F,F]\subset F\subset D^1$).
\end{dfntn}

It is often assumed that, for a special multi-flag, the Cauchy characteristics module $L(D^j)\subset D^{j+1}$ for $j=1,2,\hdots,r-1$. However, this is implied by the two conditions given in Definition \ref{Multi_flag_definition} (see Proposition 1.3 in \cite{Adachi_Jiro} or Corollary 6.3 in \cite{SHIBUYA2009793}).

\subsection{Singularity classes of multi-flags}
In this section, we review the classification of multi-flags and construct normal forms visualizing flags from a given class based on \cite{Mormul_2023_multi} which was rooted in the previous work on normal forms and 2-flags \cite{Mormul2004MultidimensionalCP, Mormul_2009Singularity2flags}.
Singularity classes of $m$-flags, also called \emph{Extended Kumpera Ruiz classes} extending definition of classes of Goursat flags, will be labeled by words over the alphabet
$$ \{1,2,\hdots, m,m+1\}$$
constrained by conditions under the name \emph{least upward jumps rule} (lujr) \cite{Mormul_2023_multi}.
\begin{dfntn}\label{def_lujr}
    A code made of letters $ \{1,2,\hdots, m,m+1\}$ satisfies the \emph{least upward jumps rule} if it starts with $1$, and any letter $l>1$ can appear only if the letter $l-1$ appeared before. Equivalently, $j_1j_2\hdots j_r$ is lujr if $j_1=1$ and $j_{l+1}\leq 1+\max(j_1,\hdots, j_l) $ for $l=1,2,\hdots, r-1$.
\end{dfntn}
To clarify the definition, we note that the code $1 \ 2 \ 1$ is lujr, while $1 \ 3 \ 1$ is not because of the appearance of the letter $3$ which is not preceded by the letter $2$.

In Appendix \ref{app:combi}, we derive the number of codes (and hence the number of singularity classes) of length $r$ 
    \begin{align}
       N(m,r)=\sum_{k=0}^{m}\left\{\frac{(m+1-k)^r}{(m+1-k)!}\cdot \sum_{n=0}^k\frac{(-1)^n}{n!}\right\}. 
    \end{align}

Let us now describe what singularity classes are. We start by describing classes of $2-$flags which is the simplest example of multi-flags and were classified first \cite{Mormul_2009Singularity2flags}. To each distribution we shall assign the word over the alphabet $\{1,2,3\}$. Later we will outline how this procedure generalizes to higher ranks.

The conditions from Definition \ref{Multi_flag_definition} lead to the following sandwich diagram \cite{Montgomery2001GeometricAT}
\begin{align}
 \begin{array}{@{}ccccccccccccc}
TM = D^0 & \supset & D^1 & \supset & D^2 & \supset & \cdots & \supset
& D^{r-1} & \supset & D^r & & \\
& & \cup & & \cup & & & & \cup & & \cup & & \\
& & F & \supset & L(D^1) & \supset & \cdots & \supset & L(D^{r-2})
& \supset & L(D^{r-1}) & \supset & L(D^r) = 0.
\end{array}
\end{align}
 All vertical inclusions in this diagram are of codimension one while all drawn horizontal inclusions are of codimension $m$. 
 Here, $L(D)$ denotes the \emph{Cauchy characteristic module} of vector fields
     \begin{align}
        L(D) = \{X\in D \ : \ [X,Y]\in D,  \ \forall \ Y\in D \}.
    \end{align}

The first partition of all germs of $2$-flags (and also more generally for $m$ flags) will be into $2^{r-1}$ \emph{sandwich classes}, labeled by words of the length $r$ over $\{1,\underline{2}\}$ and depending on a sandwich position. Each word starts with $1$. The second letter is $\underline{2} $ iff $D^2\subset F$, and $j$-th letter is $\underline{2} $ iff $D^j\subset L(D^{j-1})$. This definition makes sandwich classes pairwise disjoint. They are also nonempty (see Remark 3 of \cite{Mormul_2009Singularity2flags}).

Next, we refine sandwich classes into finer singularity classes. For that, we have to replace $\underline{2}$ with either $2$ or $3$. Let $j_s=\underline{2}$ be the first (from the right in the sandwich diagram) appearance of the letter $\underline{2}$. Let the next $\underline{2}$ be at $j_\nu$ for some numbers $1<\nu<s$.\footnote{$\nu>1$ because of the lujr property.} These are separated by $l=s-\nu -1$ letters $1$.
The core of the construction consists of taking the small flag of flag’s member $D^s$
\begin{align}
    D^s = V_1\subset V_2\subset V_3\hdots,
\end{align}
$V_{i+1}=V_i+[D^s,V_i]$ and then focusing on this new flag’s member $V_{2l+3}$. The letter $j_\nu=\underline{2}$ is specified to $3$ if and only if 
\begin{align} F\supset V_{2l+3} \text{ (when $\nu =2$)} \quad \text{or} \quad L(D^{\nu-2})\supset V_{2l+3} \text{ (when $\nu>2$).}
\end{align}
In this way, all $j_\nu=\underline{2}$ ($\nu>s$) are specified to be either $2$ or $3$.

The case of general $m$ is more subtle and lengthy to describe, so we recommend the original source \cite{Mormul_2023_multi}. Sandwich classes $\mathfrak{w}_2$, with two characters in the code, are replaced by finer classes $\mathfrak{w}_3$, with three characters, and so on, so we consider characters $k$ and larger when passing from $\mathfrak{w}_{k-1}$ to  $\mathfrak{w}_k$. In essence, one repeats operation similar to $2$-flag case and replaces $j_\nu=k$ with $k+1 $ iff 
\begin{align}  \sigma_k[\nu]\left(D^\nu \right)\subset L\left(D^{\alpha_k(\nu)-2} \right) \text{ when ($\alpha_k(\nu)>2$)} \  \text{or}  \ \sigma_k[\nu]\left(D^\nu \right)\subset F\text{ when ($\alpha_k(\nu)=2$)},
\end{align}
where $\alpha_k(\nu)$ is defined as $\alpha_k(\nu):=\alpha_{k-1}(\mu)$, with $j_\mu=k$ being the nearest $k $ to the left of $j_\nu$, and $\alpha_2$ is the distance from the previous $2$. $\sigma_k[\nu]$ is defined as $\sigma_k[\nu]:=\sigma_{k-1}[\mu]\circ \sigma_2[s_1]\circ \hdots \circ \sigma_2[s_\rho]\circ \sigma_2[\nu]$, with $\mu<s_1<s_2<\hdots<s_\rho<\nu$ such that $j_{s_u}$ are all letters  larger than $1$ and smaller than $k$, and we have $\sigma_2[\nu](U):= (2(\nu-\alpha_2(\nu)-1)+3)$-rd term of $U$'s small flag. In this way, one refines sandwich classes into singularity classes 
\begin{align}
\mathfrak{w}_2\rightarrow\mathfrak{w}_3\rightarrow\hdots\rightarrow\mathfrak{w}_{k}\rightarrow\mathfrak{w}_{k+1} \hdots\rightarrow\mathfrak{w}_{m+1}.
\end{align}

\subsection{Extended Kumpera-Ruiz normal forms}
We introduce Extended Kumpera-Ruiz (EKR) normal forms and label them by letters over the same alphabet as we used for singularity classes
$$ \bf\{1,2,\hdots, m,m+1\}.$$
EKR normal forms are polynomially written germs of rank $(m+1)$ distributions at $0\in \mathbb{R}^{(r+1)m+1}$, that were constructed in \cite{Mormul2004MultidimensionalCP}. Let us now show how to inductively construct a normal form corresponding to the code $ \j_1\j_2\hdots \j_r$.
We start with an empty label distribution
\begin{align}
    (\partial_0,\partial_1,\hdots,\partial_m):=\left(\frac{\partial }{\partial x^0_0},\frac{\partial}{\partial x^0_1},\hdots,\frac{\partial}{\partial x^0_m}\right)
\end{align}
defined in the vicinity of $0\in \mathbb{R}^{m+1}(t, x^0,\hdots, x^m)$. 

Now, let $\j_l=\kk$ be the $l$-th letter. The prefix of the code $ \j_1\j_2\hdots \j_{l-1}$ defines a distribution $D=(Z_1,Z_2,\hdots,Z_{m+1})$ around $0\in \mathbb{R}^{lm+1}(y_1,\hdots,y_{lm+1})$.
We define a new distribution
\begin{align}
    D'=\mathbf{k}D=(Z_1',Z_2',\hdots,Z_{m+1}')
\end{align}
 generated by vector fields
\begin{align}
    Z_1'&=x^l_1Z_1+\hdots +x^l_{k-1}Z_{k-1} +Z_k+(x_k^l+c_k^l)Z_{k+1}+\hdots+(x_m^l+c_m^l)Z_{m+1},\label{Z_construction_start}\\
    Z_2'&=\partial_{x^l_1},\\
    &\hdots \notag\\
    Z'_{m+1}&=\partial_{x^l_m},\label{Z_construction_finish}
\end{align}
defined around $0\in \mathbb{R}^{(l+1)m+1}(y_1,\hdots,y_{lm+1}, x^l_1,\hdots, x^l_m)$. Here $c_j^l$ are constant coefficients determined by the original distribution whose form we consider or a geometric class possessing such normal form (we refer to Section 4.1 in \cite{Mormul_2009Singularity2flags} for more details on how these constants are assigned).

If we put all constants $c_j^l$ to zero, not only the Lie algebra $L(Z_1,Z_2,\hdots,Z_{m+1})$ is nilpotent, but also the distribution $D=(Z_1,Z_2,\hdots,Z_{m+1})$ is strongly nilpotent \cite{Mormul2004MultidimensionalCP}. This is because the distribution is homogeneous and weights of all vector fields spanning it are equal $-1$ (as in \eqref{homogeneous_summands}). The absence/appearance of nonzero constants is analogous to the tangential/nontangential distinction of Goursat classes.

In what follows we put all $c_j^i$ to zero so that we shall consider only homogeneous distributions. This restricts us to strongly nilpotent distributions. General classification would involve considering also distributions with nonzero constants but then we could not extract the exact weight from normal forms. As we shall explain below, the distribution constructed by
\begin{align}
 Z_1'&=x^l_1Z_1+\hdots x^l_{k-1}Z_{k-1} +Z_k+x_k^lZ_{k+1}+\hdots+x_m^lZ_{m+1}    
\end{align}
in place of \eqref{Z_construction_start} is strongly nilpotent and we shall consider only such distributions.
\begin{nt}
    We consider $m$-flags with vanishing constants. These are strongly nilpotent distributions. However, it is not known if all strongly nilpotent EKR normal forms (or multi-flags with such normal forms) have vanishing constants. We come back to this issue in Section \ref{sec:6}.
\end{nt}

For normal forms, one can explicitly read the elements of the sandwich diagram. The lower row reads (see Proposition 2 of \cite{Mormul_2009Singularity2flags})
\begin{align}
    F&=\left(\frac{\partial }{\partial x^1_1},\frac{\partial}{\partial x^1_2},\hdots,\frac{\partial}{\partial x^1_m},\frac{\partial }{\partial x^2_1},\hdots,\frac{\partial }{\partial x^r_m}\right), \\
    L(D^j)&= \left(\frac{\partial }{\partial x^{j+1}_1},\frac{\partial}{\partial x^{j+1}_2},\hdots,\frac{\partial }{\partial x^r_m}\right) \qquad \text{ for } 1\leq j\leq r-1,\\
    L(D^r)&=(0).
\end{align}

The most important property of EKR normal forms, that we study, is that they faithfully represent singularity classes. Singularity classes of multi-flags are in one-to-one correspondence with EKR representing them \cite{Mormul_2023_multi}.
\begin{thrm}[{\cite[Theorem 2.1]{Mormul_2023_multi}}]\label{2023thm_unique}
Let $D$ be the germ at $p\in M$ of a rank-$(m + 1)$ distribution
generating a special $m$-flag of length $r\geq 1$, belonging to a fixed singularity
class $j_1.j_2...j_r$. Then the EKR pseudo-normal forms of $D$ subject to the least upward jumps rule are uniquely of the type $\bf j_1.j_2...j_r$.

In other words, the singularity class of a germ of a special $m$-flag which
is already given in an EKR form $\bf j_1.j_2...j_r$ subject to the least upward jumps rule, is $j_1.j_2...j_r$.
\end{thrm}
Next, we shall study weights of $m$-flags using their EKR presentation.

\section{\texorpdfstring{Weights of strongly nilpotent $m$-flags}{Weights of strongly nilpotent m-flags}} \label{sec:4}
\subsection{Determining weights from codes}
Let us consider weights, defined in \eqref{weight_def}, in the context of an EKR normal form $D=(Z_1,Z_2,\hdots,Z_{m+1})$ with vanishing constants. First, we put 
\begin{align}
w(Z_1)=w(Z_2)=\hdots=w(Z_{m+1})=-1,
\end{align}
where $Z_j:=Z_j^r$ are generators of the last distribution associated with the code $\j_1,\j_2,\hdots,\j_r$.  Then, by \eqref{weight_def}
\begin{align}
    w(x_1^r)= w(x_2^r)=\hdots  =w(x_m^r)=1.
\end{align}
The condition $w(Z_1)=-1$ implies that $w(Z_1^{r-1})=-1-w(x_1^r)=-2$. More generally, any step of the construction \eqref{Z_construction_start}-\eqref{Z_construction_finish} give a recursive condition on weight of vector fields $Z^s_i$ and coordinates appearing in the construction
\begin{align}
w\Big(Z_1^{s-1}\Big)&=w(Z_1^s)+w\left(\frac{\partial}{\partial x_{1}^{s}}\right)\label{weights_of_coordinates_start}, \\
    w\left(\frac{\partial}{\partial x_l^{s-1}}\right)&=w(Z_1^s)+w\left(\frac{\partial}{\partial x_{l+1}^{s}}\right),  & \text{ for } l=2,3,\hdots, j_{s}-2,\\
        w\left(\frac{\partial}{\partial x_{j_s-1}^{s-1}}\right)&=w(Z_1^s),\\
     w\left(\frac{\partial}{\partial x_l^{s-1}}\right)&=w(Z_1^s)+w\left(\frac{\partial}{\partial x_{l}^{s}}\right),  & \text{ for } l=j_{s},\hdots m.\label{weights_of_coordinates_end}
\end{align}
where we used \eqref{Z_construction_start} to obtain $w(Z^{s-1}_1)=w(Z^{s}_1)-w(x^s_1)$. The growth of weights $ w\left(\frac{\partial}{\partial x_l^{s}}\right)$ captures the nonholonomy of a distribution and is encoded in the small growth vector. In what follows, we are going to put the above recursion in terms of certain linear operators.

We define a linear automorphism $\overline{\j}$ of $\mathbb{R}^{m+1}$ associated to any fixed operation  $\j \in \{\mathbf{1},..., \mathbf{m+1}\}$, by its basis $e_0,e_1,..., e_{m}$
\begin{align}\label{linear_j}
    \overline{\mathbf{j}}(e_0)=e_0+e_1+\cdots+e_m,\qquad
\overline{\mathbf{j}}(e_i)=
\begin{cases}
e_{i-1}, & 1\le i\le j-1,\\
e_i, & j\le i\le m,
\end{cases}
\end{align}
so, in particular, $\overline{\mathbf{1}}(e_1)=e_1,...,\overline{\mathbf{1}}(e_{m})=e_{m}$ and $\overline{\mathbf{m}}(e_2)=e_1,...,\overline{\mathbf{m+1}}(e_{m})=e_{m-1}$.
This operation expresses nilpotency orders of special $m$-flags, which are strongly nilpotent.
It can be seen when the operator $\overline{\j}$ acts on a vector containing weights of $Z_1^s$ and $m$ coordinates, it produces weights of next $m$ coordinates as \eqref{weights_of_coordinates_start}-\eqref{weights_of_coordinates_end} read
\begin{align}
   \overline{\j}  \begin{bmatrix}
           -w(Z_{1}^s) \\
           w(x^s_1) \\
           \vdots \\
           w(x^s_{m})
         \end{bmatrix} =  \begin{bmatrix}
           -w(Z_{1}^{s-1}) \\
           w(x^{s-1}_1) \\
           \vdots \\
           w(x^{s-1}_{m})
         \end{bmatrix}.
\end{align}
 Additionally, because weights of coordinates are positive, $|w(Z_1^s)|$ is equal to the smallest weight in the new vector. Therefore, taking the sequence of all weights except the smallest one, we take all weights $w(Z_{1}^l),
           w(x^l_1),
           \hdots,
           w(x^l_{m})$ ($l=1,2,\hdots, r$) once but we omit exactly once $w(Z_{1}^r),\hdots,w(Z_{1}^1)$. This is equivalent to taking the sequence of weights $ w(x^r_1),\hdots w(x^r_m),w(x^{r-1}_1), \hdots, w(x^{r-1}_m),\hdots w(x^{0}_0),\\w(x^{0}_1), \hdots, w(x^{1}_m)$.

From weights, one can, in particular, derive nilpotency order of Lie algebra spanned by vector fields $Z_i$.
\begin{thrm}[{\cite[Theorem 4]{Mormul2004MultidimensionalCP}}]\label{thm:multiflags_and_nilpotency}
     Every rank $m+1$ distribution $D$ generating a special $m$-flag of a length
 $r\geq 1$ is locally weakly nilpotent
 and a local nilpotent basis around a point $p$ is given by the EKR normal form $\{Z_1,...,Z_{m+1}\}$.
 The nilpotency order of the real Lie algebra $L(\j_1\j_2 ... \j_r):=  L_\mathbb{R}(Z_1,...,Z_{m+1})$ is equal to the (last) component of $e_{m}$ in the vector $ \overline{\j}_1
\overline{ \j}_2...\overline{\j}_r(e_0 + ... +e_{m})$
 Moreover, if a germ of $D$ at a certain point admits a local EKR form with no non-zero
 constants, then that germ is also strongly nilpotent in the sense of \cite{Agrachev2001ONTS} and \cite{MomrulGoursatdim7}.
\end{thrm}
It was conjectured \cite{Mormulnotexceedingsix, Mormul2004MultidimensionalCP} that, for a given distribution $D$, the nilpotency orders are minimal in the EKR presentation. This remains an open question.

The Theorem \ref{thm:multiflags_and_nilpotency} allows us to express the nonholonomy degree of a given strongly nilpotent distribution $D$ in terms of its EKR form.
\begin{prpstn}[\cite{Mormul2004MultidimensionalCP},  Corollary 2]
   For distribution germs $D$ as in the ‘Moreover’ part of Theorem \ref{thm:multiflags_and_nilpotency}, one is able to effectively compute the degree of nonholonomy of $D$. It is equal to the nilpotency order of the nilpotent Lie algebra given by the nilpotent approximation of $D$, and in the discussed case that algebra is just $ L_\mathbb{R}(Z_1,...,Z_{m+1})$. Hence, the nonholonomy degree of such $D$ equals the last component in the vector $\overline{ \j}_1...\overline{\j}_r(e_0 + ... +e_{m})$.
\end{prpstn}

Before moving to the main theorem, let us demonstrate an example of how linear operations introduced in \eqref{linear_j} work.
\begin{xmpl}
    Let $m=2$. The nilpotency order of the algebra $L(\mathbf{1}.\mathbf{1}.\mathbf{2})$ can be computed from $\overline{ \mathbf{1}}\overline{ \mathbf{1}}\overline{\mathbf{2}}$ acting on $e_0+e_1 + e_2$
    \begin{align}
        \overline{ \mathbf{1}} \overline{ \mathbf{1}} \overline{\mathbf{2}} \begin{bmatrix}
            1 \\ 1 \\ 1
        \end{bmatrix} = \begin{bmatrix} 1 & 0 & 0 \\ 1 &1 & 0 \\ 1&0&1\end{bmatrix}
        \begin{bmatrix} 1 & 0 & 0 \\ 1 &1 & 0 \\ 1&0&1\end{bmatrix}
        \begin{bmatrix} 1 & 1 & 0 \\ 1 &0 & 0 \\ 1&0&1\end{bmatrix}\begin{bmatrix}
            1 \\ 1 \\ 1
        \end{bmatrix} = \begin{bmatrix}
            2 \\ 5 \\ 6
        \end{bmatrix},
    \end{align}
so it is equal to $6$.

Similarly, the nilpotency order of the algebra $L(\mathbf{1}.\mathbf{2}.\mathbf{2})$ 
\begin{align}
    \overline{ \mathbf{1}}\overline{ \mathbf{2}}\overline{\mathbf{2}}
    \begin{bmatrix} 1 \\ 1 \\ 1\end{bmatrix}=
    \begin{bmatrix} 1 & 0 & 0 \\ 1 &1 & 0 \\ 1&0&1\end{bmatrix}
        \begin{bmatrix} 1 & 1 & 0 \\ 1 &0 & 0 \\ 1&0&1\end{bmatrix}
        \begin{bmatrix} 1 & 1 & 0 \\ 1 &0 & 0 \\ 1&0&1\end{bmatrix}\begin{bmatrix} 1 \\ 1 \\ 1\end{bmatrix}=
    \begin{bmatrix} 3 \\ 5 \\ 7\end{bmatrix}
\end{align}
is equal to $7$.
\end{xmpl}

\subsection{Inequivalence of EKR forms}

In this section, we present the main theorem of this paper.
\begin{thrm}\label{main_theorem}Fix $m\ge 1$ and $r\ge 1$. 
Let $D$ and $\widetilde D$ be two \emph{strongly nilpotent} germs of special $m$-flags of length $r$,
presented in homogeneous EKR normal forms (all constants $c^l_j=0$), with lujr codes
$C=j_1\cdots j_r$ and $\widetilde C=\tilde j_1\cdots \tilde j_r$.
If $C\neq \widetilde C$, then the associated small growth vectors satisfy
\begin{align}
    \mathrm{sgrv}(D)\neq \mathrm{sgrv}(\widetilde D),
\end{align}
hence the germs $D$ and $\widetilde D$ are not diffeomorphic.

In particular, within the homogeneous (strongly nilpotent) locus, distinct singularity classes
are pairwise inequivalent. \end{thrm}

The existence is implied from the stratification into geometric classes and the construction of EKR normal forms, where we can match $ \j_1 \j_2 . . . \j_r$ form to  $ j_1 j_2 . . . j_r$ class and choose constant appropriately. We look at strongly nilpotent geometric classes through EKR lenses and show that these are indeed disjoint and have different weights (hence also small growth vectors in the strongly nilpotent case).
In the previous section, we stated that elements of the small growth vector (dimensions of Lie flag) of EKR normal forms are encoded by weights being taken as $m$ out of $m+1$
components of the following vectors
\begin{align}
    \begin{bmatrix} 1 \\ \hdots \\ 1\end{bmatrix}, \ \overline{\mathbf{j}_r}
    \begin{bmatrix} 1 \\ \hdots \\ 1\end{bmatrix}, \quad \overline{ \mathbf{j}_{r-1}}\overline{\mathbf{j}_r}
    \begin{bmatrix} 1 \\ \hdots \\ 1\end{bmatrix}, \ \hdots,\quad \overline{\mathbf{j}_{1}}\hdots\overline{\mathbf{j}_{r-1}}\overline{\mathbf{j}_r}
    \begin{bmatrix} 1 \\ \hdots \\ 1\end{bmatrix}.
\end{align}
Since each singularity class admits its lujr EKR representative and the small growth vector is a diffeomorphism invariant, it suffices to show that distinct lujr codes yield distinct weight sequences in the homogeneous (strongly nilpotent) case.

To prove this fact, it will be convenient for us to consider a more general problem that can be proven inductively. First, we introduce auxiliary notation.
For a given $(m+1)$-dimensional vector $v=(v_0,v_1,\hdots,v_m)$, let $v^{(s)}$ denote its $s$-set of vectors, so that
\begin{align}\label{s_set_of_vectors_before}
    v^{(s)} =\left\{w=(v_0,v_1,..,v_{s-1},v_{s-1+\sigma(1)},\hdots,v_{s-1+\sigma(m-s+1)} )\ | \ \sigma\in \Sigma_{m-s+1}\right\}.
\end{align}
This notion is expanded in Section \ref{subsec:ssets}.
\begin{thrm}\label{main_theorem_rephrased}
    Let $v,w$ be two $(m+1)$-dimensional vectors with positive integer components such that $v^{(k)}=w^{(k)}$. Then sequences of weights associated with two lujr codes $ j_1 j_2 . . . j_rk $ and $ j'_1 j'_2 . . . j'_rl$, ending with $k$ and $l>k$ respectively,
    \begin{align}
      \overline{\j_1 \j_2}\hdots\overline{ \j_r \kk} v,   \quad  \overline{\j_1' \j_2'}\hdots \overline{\j_r' \ll} w 
    \end{align}
    are different. Here, by sequences of weights, we mean the sequence of $m$ largest components of 
    \begin{align}
    \overline{ \mathbf{k}}v,\ \overline{\mathbf{j}_r}\overline{ \mathbf{k}}
    v, \ \overline{ \mathbf{}}\overline{ \mathbf{j}_{r-1}}\overline{\mathbf{j}_r}\overline{ \mathbf{k}}
    v, \ \hdots,\ \overline{\mathbf{j}_{1}}\hdots\overline{\mathbf{j}_{r-1}}\overline{\mathbf{j}_r}\overline{ \mathbf{k}}
    v, \quad \text{ and } \quad \overline{ \mathbf{l}}w,\ \overline{ \mathbf{}\mathbf{j}'_r \mathbf{l}}
    w, \ \overline{ \mathbf{}}\overline{ \mathbf{j}'_{r-1}}\overline{\mathbf{j}'_r} \overline{ \mathbf{l}}
    w, \ \hdots,\ \overline{\mathbf{j}'_{1}}\hdots\overline{\mathbf{j}'_{r-1}\mathbf{j}'_r \mathbf{l}}
    w.
\end{align}
\end{thrm}

In the Appendix \ref{app:proof} we shall introduce numerous definitions to prove Theorem \ref{main_theorem_rephrased} and show how it implies Theorem \ref{main_theorem}.

\section{Weights of tangential Goursat flags}\label{sec:5}
After the general result of the previous section, we restrict our attention to the Goursat case.
Goursat distributions have been investigated for more than a century now since the works of E. von Weber \cite{vonWeber} and E. Cartan \cite{CartanSurLA}. They owe their name to E. Goursat, who wrote a treatise \cite{Goursat1923LeonsSL} which popularized those objects. 
The seminal work of Kumpera and Ruiz \cite{giaro1978lectureKR, KumperasurlequivalencedePfaff} showed that Goursat distributions exhibit singularities. The first geometric characterization of these singularities can be found in \cite{Bryant1993RigidityOI}. These singularities are geometric and can be encoded in the small growth vector \cite{Jean_1996}, later sharpened in \cite{SGV_Mor_2004}. When it comes to the characterization by a small growth vector, the great simplification was brought by Jean and the recursions he introduced using a kinematic model of a car with trailers \cite{Jean_1996}.

In the language of special multi-flags, Goursat distributions are $1$-flags with the restriction that distinct codes start with \emph{two letters} $1$ due to Engel’s theorem \cite{engel1890}. In this case, EKR normal forms reduce to, so-called, Kumpera-Ruiz (KR) normal forms \cite{giaro1978lectureKR, KumperasurlequivalencedePfaff} whose polynomial construction is well-known and can be described inductively. The construction of \eqref{Z_construction_start}-\eqref{Z_construction_finish} after the appearance of the letter ${\bf k}\in \{1,2\}$ simplifies to
\begin{align}
    Z_1'&=\begin{cases}Z_1+(x+c^l)Z_{2},  &\text{if } {\bf k}=1,\\
    xZ_1+Z_{2},  &\text{if } {\bf k}=2,
    \end{cases}\label{Goursat_construction}\\
    Z_2'&=\partial_{x},
\end{align}
defined around $0\in \mathbb{R}^{l+1}(y_1,\hdots,y_{l+1},x)$. First, we note that the constants $c^l$ in \eqref{Goursat_construction} associated with the prefix consisting of first letters ${\bf 1}$ can be always put zero by a polynomial transformation on the jet \cite{vonWeber}.

We restrict to tangential germs, which are strongly nilpotent \cite{SGV_Mor_2004}.
\begin{dfntn}\label{definition_tangential}
    A KR class is called \emph{tangential} if all constants $c^l$ in \eqref{Goursat_construction} vanish. Tangential (non-tangential) germs are elements of tangential (non-tangential) classes.
\end{dfntn}

To relate weights to the derived function, recall that for a rank--$2$ Goursat
distribution the small Lie flag increases dimension by at most $1$.
In adapted coordinates $(x_1,\dots,x_{r+2})$ one has
a weight sequence $(w_1,\dots,w_{r+2})$ with $w_1=w_2=1$, and the successive
differences
\begin{align}\label{eq:weights-from-derived}
    d_j := w_{j+1}-w_j,\qquad j=2,\dots,r+1,
\end{align}
form Jean's derived vector \cite{SGV_Mor_2004} (equivalently, the derived function with
multiplicities). Conversely, once the derived vector $(d_2,\dots,d_{r+1})$ is
known, the weights are recovered by the recurrence
\begin{align}
w_1=w_2=1,\qquad w_{j+1}=w_j+d_j, \qquad j=2,\dots,r+1.
\end{align}
For tangential germs the KR coordinates are privileged/adapted in the sense of
Bellaïche, so this reconstruction applies directly to the KR normal form.

KR normal forms of tangential germs are characterized only by their encoding. Following the notation of \cite{Mormul2011SmallGV} adapted to our setting, let us consider a KR normal form with a code having $k_s$ letters ${\bf 1} $ at the beginning, then a letter ${\bf 2}$, then $k_s$, and so on so that it has $s+1$ letters ${\bf 2}$:
\begin{align}
    {\bf 1}^{k_{s+1}} {\bf 2}\ {\bf 1}^{k_{s}}{\bf 2} \hdots {\bf 2}\ {\bf 1}^{k_{0}}\end{align}
where $k_j\geq 0$ for $j=0,1,\hdots,s$, and $k_{s+1}\ge 2$.
Proposition 3.2 of \cite{SGV_Mor_2004} states that the derived function takes the form
\begin{align}\label{eqn:der_function_def}
    \mathrm{der} = (\underbrace{A_0,\hdots A_0}_{2+k_0 \text{ times}},\underbrace{A_1,\hdots A_1}_{1+k_1\text{ times}},A_2,\hdots A_{s},\underbrace{A_{s+1},\hdots A_{s+1}}_{ k_{s+1}-1 \text{ times}}),
\end{align}
where components $A_j$ can be determined from the following theorem from \cite{SGV_Mor_2004}, which was proven in  \cite{Mormul2011SmallGV}.
\begin{thrm}[{\cite[Theorem 3.3]{SGV_Mor_2004}}]\label{An_noncompact_form_of_thm2}
    The components of the derived function $A_0,A_1,A_2,A_3,\hdots, A_s, A_{s+1}$ in \eqref{eqn:der_function_def} appear with the multiplicities 
    \begin{align}
        2+k_0,\ 1+k_1,\ 1+ k_2, \hdots, 1+k_s, \ \text{ and } \ k_{s+1}-1 \quad \text{times},
    \end{align}
    where the sequence $A_j$ is determined by the following recursion
    \begin{align}\label{A_recursion_bezG0}
        A_0&=1,\\
        A_1&=2+k_0,\label{A_recursion_bezG1} \\
        A_j&=A_{j-2}+(1+k_{j-1})A_{j-1}\label{A_recursion_bezG2}.
    \end{align}
\end{thrm}
The important feature of this recursion is that it can be solved to obtain formulas for $A_j$ only in terms of the sequence $k_j$, a result that we prove in Appendix \ref{app:recursion}.
\begin{thrm}\label{An_compact_form_of_thm2}
The recursion \eqref{A_recursion_bezG0}-\eqref{A_recursion_bezG2} is solved by the following sequence
    \begin{align}\label{Aj_formulas}
        A_j = 1+\sum_{{\substack{0\leq i_1<...<i_p<j\\i_1\not\equiv i_2,...,i_{p-1}\not\equiv i_p,i_p\not\equiv j(\text{mod }2)}}} \prod_{l=1}^p(1+k_{i_l}).
    \end{align}
\end{thrm}
\begin{nt}
    Analogous recursive relations \cite[Theorem 3.4]{SGV_Mor_2004} and formulas can be generalized also to the non-tangential case. This is, however, much more laborious and formulas become more complicated, therefore we don't treat it in this work focused on the strongly nilpotent case.
\end{nt}

The closed formula in Theorem \ref{An_compact_form_of_thm2} also yields explicit formulas for the coordinate weights. Indeed, by~\eqref{eq:weights-from-derived}, the weights
are obtained by taking cumulative sums of the entries of the derived vector.
Let $\mu_j$ denote the multiplicity of $A_j$ in \eqref{eqn:der_function_def}, so that
\begin{align}
\mu_0 &= 2+k_0, \\
\mu_j &= 1+k_j, \qquad 1\leq j\leq s, \\
\mu_{s+1} &= k_{s+1}-1.
\end{align}
For $0\leq n\leq s+1$, define
\begin{align}
q_n := \sum_{j=0}^{n}\mu_j,
\qquad
S_n := \sum_{j=0}^{n}\mu_j A_j.
\end{align}
Then $S_n$ is the sum of all entries in the
first $n+1$ blocks of the derived vector. Consequently,
\begin{align}\label{eq:weights_in_S1}
w_{q_n+2}=1+S_n.
\end{align}
More generally, the weights within the block corresponding to $A_n$ are
given by
\begin{align}\label{eq:weights_in_S2}
w_{q_{n-1}+2+t}=1+S_{n-1}+tA_n,
\qquad
1\leq t\leq \mu_n,
\end{align}
and $w_1=w_2=1.$    
Thus, the formulas above together with Theorem \ref{An_compact_form_of_thm2}, lead to explicit expressions for
all coordinate weights.
\begin{thrm}\label{thm:sums_of_Goursats}
    For $0\leq n\leq s$, the sum $S_n$ of all entries of the derived vector corresponding to $A_0,\ldots,A_n$ in the derived vector is
    \begin{align}\label{sum_less_s}
        S_n = 1+\sum_{{\substack{0\leq i_1<...<i_p\leq n\\i_1\not\equiv i_2,...,i_{p-1}\not\equiv i_p(\text{mod }2)}}} \prod_{l=1}^p(1+k_{i_l}).
    \end{align}
The sum of all entries of the derived vector is
       \begin{align}
        S_{s+1}=k_{s+1} +k_{s+1}\cdot \sum_{{\substack{0\leq i_1<...<i_p<s+1\\i_1\not\equiv i_2,...,i_{p-1}\not\equiv i_p,i_p\not\equiv s+1(\text{mod }2)}}} \prod_{l=1}^p(1+k_{i_l}) + \sum_{{\substack{0\leq i_1<...<i_p<s+1\\i_1\not\equiv i_2,...,i_{p-1}\not\equiv i_p,i_p\equiv s+1(\text{mod }2)}}} \prod_{l=1}^p(1+k_{i_l}).
    \end{align}
\end{thrm}
Theorem \ref{thm:sums_of_Goursats} together with equations \eqref{eq:weights_in_S1}-\eqref{eq:weights_in_S2} and Theorem \ref{An_compact_form_of_thm2} gives exact formulas for weights.

\section{The question of strong nilpotency}\label{sec:6}
We have shown that, within the homogeneous EKR normal forms, the
code of a singularity class is distinguished by the weights: different lujr codes give
different weight sequences, hence different small growth vectors and
inequivalent germs. In the Goursat case, we have also obtained explicit
formulas for the weights of tangential classes, building on the work of
Jean and Mormul. These results clarify the strongly nilpotent part of
the theory, but they also point to a remaining question: how large is
the homogeneous locus inside the full strongly nilpotent locus?

In this paper we have restricted to EKR normal forms in which all
constants $c^l_j$ appearing in the normal form \eqref{Z_construction_start} vanish\footnote{As we commented before, constants corresponding to the prefix of 1's can be put to zero without any loss of generality.}.
When all these constants are zero, the generators $Z_1,\ldots,Z_{m+1}$ are homogeneous with respect to the weights, and therefore coincide with their
nilpotent approximations. Hence, the corresponding distribution germ is strongly nilpotent. Thus the vanishing of all constants is a sufficient condition for strong nilpotency.

The converse implication is the open point. We formulate it as the following conjecture.
\begin{cnjctr}\label{con:multi}
A special $m$-flag germ in the singularity class $\underbrace{{\bf 1 1... 1}}_{l-1}\underbrace{\j_l\j_{l+1} }_{j_l>1}... \j_r $ is strongly nilpotent if and only if
\begin{equation}
\begin{aligned}
c^{l}_{j_l} &= \cdots = c_m^l = 0,\\
c^{l+1}_{j_{l+1}} &= \cdots = c_m^{l+1} = 0,\\
&\vdots\\
c^{r}_{j_r} &= \cdots = c_m^r = 0.
\end{aligned}
\label{eq:vanishing-c}
\end{equation}
\end{cnjctr}
The $ \impliedby$ part was proven in \cite{Mormul_2004}. In fact, this claim is still conjectural even for Goursat flags that have been studied for many decades now. We rephrase the still unproven conjecture of \cite{MomrulGoursatdim7}.
\begin{cnjctr}\label{con:Goursat}
Goursat germs are strongly nilpotent if and only if they are tangential.
\end{cnjctr}
So far all examples confirm the Conjecture \ref{con:Goursat} \cite{Mormulnotexceedingsix, Mormul_DoModuli, MomrulGoursatdim7} but one would probably need new methods to prove the full assertion.  The case of general multi-flags would be even more complicated as, even in $m=2$, the structure of local moduli is surprising in many ways \cite{mormul2024localmodulispecial2flags}. We hope that results such as \cite{Mormulnotexceedingsix, Mormul_DoModuli, MomrulGoursatdim7} point in the right direction and our results are valid for all strongly nilpotent multi-flags.

\appendix

\section{Number of codes satisfying least upward jump rule}\label{app:combi}
In this appendix, we investigate code labeling singularity classes. We determine the number of words satisfying the least upward jumps rule (lujr) defined in Definition \ref{def_lujr}. A code made of letters $ \{1,2,\hdots, m,m+1\}$ is said to satisfy the least upward jumps rule if it starts with $1$, and any letter $l>1$ can appear only if the letter $l-1$ appeared before.
This allows us to determine the number of singularity classes that are encoded by such words.

First, we can consider codes over the two-letter alphabet $\{1,2\}$. The lujr rule is always satisfied if the word starts with 1, so the number of lujr words of the length $r$ is
\begin{align}
    N(1,r) = 2^{r-1}.
\end{align}
This is closely related to the number of Kumpera-Ruiz classes of Goursat flags which is $ 2^{r-2}$. The discrepancy is due to the Engel theorem for Goursat distributions of corank $r$ that forces $\mathbf{j}_2=1$ which has no analog for $m>1$ (see Proposition 1 of \cite{Mormul2004MultidimensionalCP}). The formula  $N(1,r)$ is a point of departure to study general ranks, where the number of lujr codes reflects symmetries of the stratification \cite{Mormul_2023_multi}, and more transparent construction of normal forms \cite{Mormul2004MultidimensionalCP}.
The task of finding the number $N(m,r)$ of codes over $(m+1)$-letter alphabet is important to better understand the structure of the \emph{monster tower} \cite{Mormul2022MonsterTF}. It was shown in \cite{Mormul2022MonsterTF} that
\begin{align}
    N(2,r)&= \frac{1}{3!}3^r+\frac{1}{2}, &r\geq 3,\label{eqn:example_N2r}\\
    N(3,r)&=\frac{1}{4!}4^r+ \frac{1}{4}2^r+\frac{1}{3}, &r\geq 4,\\
     N(4,r)&= \frac{1}{5!}5^r+ \frac{1}{12}3^r+ \frac{1}{6}2^r+\frac{3}{8}, &r\geq 5.\label{eqn:example_N4r}
\end{align}
These numbers suggest a pattern that is clear among leading and subleading coefficients of $n^r$.  In fact, \cite{Mormul2022MonsterTF} showed that the following asymptotic holds
\begin{align}
    \lim_{r\rightarrow \infty}\frac{N(m,r)}{(m+1)^r}=\frac{1}{(m+1)!}.
\end{align}

In the following theorem, we give an explicit closed formula for the numbers $N(m,r)$.
\begin{thrm}
The number of the lujr codes of length $r$ over the $(m+1)$-letter alphabet is
    \begin{align}\label{number_of_lujr}
       N(m,r)=\sum_{k=0}^{m}\left\{\frac{(m+1-k)^r}{(m+1-k)!}\cdot \sum_{n=0}^k\frac{(-1)^n}{n!}\right\}. 
    \end{align}
\end{thrm}
\begin{proof}
    First, observe that the largest letter $k$ appearing in a given code $\mathcal{C}$ is equal to the number of different characters appearing in the code $\mathcal{C}$ because it is lujr ($\mathcal{C}$ contains all $k,k-1,\hdots,2,1$).

    The second observation is that there are $(m+1-k)!$ permutations of the set of characters $\{1,2,\hdots, m+1\}$ that leave the code $\mathcal{C}$ (with the largest character $k$) lujr. This is because any nontrivial permutation of the set $\{1,2,\hdots, k\}$ spoils the lujr property, which requires order in the first appearances of $\{1,2,\hdots, k\}$, while letters $k+1,\hdots,m,m+1$ do not appear in $\mathcal{C}$ and can be permuted in any way.

    Let us write
    \begin{align}\label{eqn-sum_with_a}
        N(m,r)=a_{m+1}+a_{m}+\hdots +a_2+a_1,
    \end{align}
    where $a_k$ is the number of the lujr words with $k$ being precisely the largest character.

    There are exactly $(m+1)^r$ words over the alphabet in question. A word with $k$ different characters is a fixed point of $ (m+1-k)!$ permutations. Every word can be transformed by a permutation of the alphabet into the lujr word. Therefore, we can count all words by elements on orbits of the permutation group
    \begin{align}
        (m+1)^r =\frac{(m+1)!}{0!}a_{m+1}+\frac{(m+1)!}{1!}a_{m}+\hdots+\frac{(m+1)!}{(m-1)!}a_{2}+\frac{(m+1)!}{m!}a_{1},
    \end{align}
    so equivalently
    \begin{align}\label{eqn:start_of_recursion_an}
        \frac{(m+1)^r}{(m+1)!}=a_{m+1}+a_m+\sum_{k=1}^{m-1}\frac{a_k}{(m+1-k)!}.
    \end{align}
    Drawing on the exemplary data \eqref{eqn:example_N2r} -\eqref{eqn:example_N4r}, we want to rewrite the sum or the RHS of \eqref{eqn-sum_with_a} as a combination of quantities $\frac{k^r}{k!}$, $k=1,2,\hdots,m+1$
\begin{align}\label{expansion_of_an_in_sn}
    a_{m+1}+a_m+\sum_{k=1}^{m-1}a_{m+1-k} = \sum_{k=0}^{m}s_k\frac{(m+1-k)^r}{(m+1-k)!}.
\end{align} 
In view of \eqref{eqn:start_of_recursion_an}, $s_0=1$ and $s_1=0$. To get hold of the remaining $s_k$'s, we expand the terms $\frac{(m+1-k)^r}{(m+1-k)!}$ as in \eqref{eqn:start_of_recursion_an} and stipulate the equality of the coefficients standing on both sides of \eqref{expansion_of_an_in_sn} at $ a_{m-1},a_{m-2},\hdots,a_{2},a_{1}$:
    \begin{align}\label{eqn-s_k_recursion}
    s_k&=1-\sum_{l=0}^{k-1}\frac{1}{(k-l)!}s_l.
\end{align}
    for, respectively, $k=2,3,\hdots, m-1,m$ and with $s_0=1,s_1=0$.
However, one can treat \eqref{eqn-s_k_recursion} as an infinite system of equations for unknown $s_k,k\geq 2$, where the dependence on $m$ disappears. There remains to solve this infinite recursive system of equations with the initial data  $s_0=1,s_1=0$. 
Introduce a generating function $S:= \sum_{n=0}^\infty s_n x^n$. The equation \eqref{eqn-s_k_recursion} can be rewritten as
\begin{align}
    x^k=s_kx^k+\sum_{l=0}^{k-1}\frac{x^{k-l}}{(k-l)!}s_lx^l.
\end{align}
Summing from $k=0 $ to infinity, we obtain a (convergent in $|x|<1$ due to $s_k\leq 1$) power series satisfying the equation
\begin{align}
    \frac{1}{1-x}= S + S\sum_{t:=k-l=1}^\infty\frac{x^t}{t!}=Se^x,
\end{align}
where we changed the order of summation between $k$ and $l$. The generating function is then 
\begin{align}
    S=\frac{e^{-x}}{1-x},
\end{align}
which allows us to write coefficients
\begin{align}
    s_k=\sum_{n=0}^k\frac{(-1)^n}{n!}.
\end{align}
With that, we obtain the formula
\begin{align}
       N(m,r)=\sum_{k=0}^{m}\left\{\frac{(m+1-k)^r}{(m+1-k)!}\cdot \sum_{n=0}^k\frac{(-1)^n}{n!}\right\}. 
    \end{align}
\end{proof}

\section{Proofs of Theorem \ref{main_theorem} and  Theorem \ref{main_theorem_rephrased}}\label{app:proof}
Before we proceed with the proof, let us introduce some auxiliary definitions that will play a role in certain steps later.
\subsection{\texorpdfstring{$s$-sets of vectors}{s-sets of vectors}}  \label{subsec:ssets}
We define $s$-set $V^{(s)}(A_0, A_1,...,A_{s-1};A_s,...A_{m})$ of vectors as a set of vectors with given values $A_0, A_1, \hdots,A_{m}$ up to a permutations of last $m+1-s$ components
\begin{align}
    v \in V^{(s)}(A_0,...,A_{s-1};A_s,...A_{m})\iff v=\begin{bmatrix}  A_0 \\A_1\\..\\A_{s-1}\\A_{s-1+\sigma(1)}\\\vdots\\A_{s-1+\sigma(l)} \\\vdots\\A_{s-1+\sigma(m-s+1)} \end{bmatrix} \text{ for a permutation }\sigma\in \Sigma_{m-s+1}.
\end{align}
For a given $(m+1)$-dimensional vector $v=(v_0,v_1,\hdots,v_m)$, let $v^{(s)}$ denote its $s$-set of vectors, so that
\begin{align}\label{s_set_of_vectors}
    v^{(s)} =\left\{w=(v_0,v_1,..,v_{s-1},v_{s-1+\sigma(1)},\hdots,v_{s-1+\sigma(m-s+1)} )\ | \ \sigma\in \Sigma_{m-s+1}\right\}.
\end{align}
For $s=0$ we interpret $v^{(0)}$ as the set of all permutations of $(v_0,\dots,v_m)$, i.e. equality $v^{(0)}=w^{(0)}$ means equality of multisets of components.

Next, we state some properties implied by this notation. 
\begin{prpstn}
Let $v$ and $w$ be two $(m+1)$-dimensional vectors. Then
    \begin{itemize}
\item $v^{(s)}=w^{(s)}$ and $v^{(s+1)}\neq w^{(s+1)}$ if and only if
\begin{align}
    v^{(s)}=w^{(s)}, \quad  v_s\neq w_s.
\end{align}
        \item $v$ and $w$ have the same set (with multiplicities) of values if and only if
\begin{align}
    v^{(0)}=w^{(0)}.
\end{align}
\item   $w\in v^{(s+1)}$ is stronger than $w\in v^{(s)}$: 
\begin{align}
    v^{(s+1)}=w^{(s+1)} \implies v^{(s)}=w^{(s)}.
\end{align}
    \item There is only one 1-element permutation
\begin{align}
     v^{(m)}= w^{(m)} \implies v^{(m+1)}=w^{(m+1)}.
\end{align}
\end{itemize}
\end{prpstn}

\subsection{Matrices of the lujr codes}
Let us look at the matrices $\overline{\j_1 \j_2 . . . \j_r}$ associated with the action of lujr codes on a reference vector $e_0+e_1+\hdots+e_m$. 
The action of $\j$ on a basis defined in \eqref{linear_j} has the matrix\footnote{To simplify notation in this section, we will use $\j$ to denote the matrix $\overline{\j}$.}
\begin{align}
    \j v=  \begin{bmatrix}
1 & 1 & 0 & 0 & 0 & 0& 0 \\
1 & 0 & \hdots & 0 & 0 & 0& 0 \\
1 & 0 & 0 & 1 & 0& 0 & 0 \\
1 & 0 & 0 & 0 & 0& 0 & 0 \\
1 & 0 & 0& 0 & 1 & 0 & 0 \\
\hdots & 0 & 0 & 0& 0& \hdots & 0 \\
1 & 0 & 0 & 0& 0 & 0 & 1 
\end{bmatrix}  
 \begin{bmatrix}
v_0 \\
\hdots \\
v_{j-2} \\
v_{j-1} \\
v_j\\
 \hdots\\
v_m 
\end{bmatrix}  
\end{align}
which can be rewritten in terms of the vector's components
\begin{align}\label{eq:lujrmatrix}
\mathbf{j}(v)_k=
\begin{cases}
v_0 +v_{k+1}, & k<j-1,\\
v_0, & k=j-1, \\
v_0+v_k, & j\leq k\leq m.
\end{cases}
\end{align}

First, observe that for any $j\in\{1,\dots,m+1\}$ and any $v=(v_0,\dots,v_m)$, the vector $\j v$ has exactly one component equal to $v_0$ (namely $(jv)_{j-1}=v_0$), and all remaining components are of the form $v_0+v_i$ with $i\in\{1,\dots,m\}$.
Hence, the multiset of the $m$ largest components of $\j v$ is
\begin{align}\label{set_of_weights}
    \{\,v_0+v_1,\ v_0+v_2,\ \dots,\ v_0+v_m\,\},
\end{align}
independent of $j$.

We prove several elementary properties that will be used in the proof.

\begin{lmm}\label{obserwacje}
Let $v$ and $w$ satisfy $v^{(s)}=w^{(s)}, v^{(s+1)}\neq w^{(s+1)}.$ Then
\begin{enumerate}
    \item $ ( \j v)^{(s-1)}=(\j w)^{(s-1)},(\j v)^{(s)}\neq (\j w)^{(s)}$ for $j > s>0$,
    \item  $(\j v)^{(s)}=(\j w)^{(s)}, (\j v)^{(s+1)}\neq (\j w)^{(s+1)}$ for $j \leq s$,
    \item if $s=0$, $(\j v)^{(0)}\neq (\j w)^{(0)}$. 
\end{enumerate}
\end{lmm}
\begin{proof}
 We may assume $j\ge 2$ (the case $j=1$ is treated similarly, with $(\mathbf{1}v)_0=v_0$). From $v^{(s)}=w^{(s)}$ we get $v_0=w_0,\dots,v_{s-1}=w_{s-1}$.
Moreover, $v^{(s+1)}\neq w^{(s+1)}$ implies $v_s\neq w_s$.

    When $s>0$, the set of values of components of $\j v, \j w $ vectors is 
    \begin{align}
        \{v_0,v_0+v_1,\hdots,v_0+v_m\}=\{w_0,w_0+w_1,\hdots,w_0+w_m\}.\end{align}

    1. $j > s>0$. For $t=0,\hdots, s-2,$
    \begin{align}
        (\j v)_t=v_0+v_{t+1}=w_0+w_{t+1}= (\j w)_t.
    \end{align}
We also have
\begin{align}
    (\j v)_{s-1}=v_0+v_{s}=w_0+v_{s}\neq w_0+w_{s}= (\j w)_{s-1},
\end{align}
therefore, $(\j v)^{(s)}\neq (\j w)^{(s)}, ( \j v)^{(s-1)}=(\j w)^{(s-1)}$.

2. $j \leq s$. Then  
    \begin{align}
        &(\j v)_0=v_0+v_1=w_0+w_1= (\j w)_0, \ \hdots ,\ (\j v)_{j-2}=v_0+v_{j-1}=w_0+w_{j-1}= (\j w)_{j-2},\\
        &(\j v)_{j-1}=v_0=w_0= (\j w)_{j-1}, (\j v)_{j}=v_0+v_{j}=w_0+w_{j}= (\j w)_{j},\hdots \\ &\hdots (\j v)_{s-1}=v_0+v_{s-1}=w_0+w_{s-1}= (\j w)_{s-1}, 
         \end{align}
so $ ( \j v)^{(s)}=(\j w)^{(s)}$, but
\begin{align}
    (\j v)_{s}=v_0+v_{s}=w_0+v_{s}\neq w_0+w_{s}= (\j w)_{s},
\end{align}
which means $(\j v)^{(s+1)}\neq (\j w)^{(s+1)}$.

3. If $s=0$ then $v^{(0)}=w^{(0)}$ and $v_0\neq w_0$. The sum of components of $jv$ equals
$\sum_i (jv)_i = m v_0 + \sum_i v_i$ (and similarly $\sum_i (jw)_i = m w_0 + \sum_i w_i$),
hence $\sum_i (jv)_i \neq \sum_i (jw)_i$. Therefore $(jv)^{(0)}\neq (jw)^{(0)}$.
\end{proof}

It is also useful to consider two different letters acting on two vectors.
\begin{lmm}\label{equal}
Let $v^{(k)}=w^{(k)}$ and let $l>k$. Then
\begin{align}
    (\kk v)^{(k-1)}= (\ll w)^{(k-1)}
\qquad\text{and}\qquad
(\kk v)^{(k)}\neq (\ll w)^{(k)}.
\end{align}
\end{lmm}
\begin{proof}
From $v^{(k)}=w^{(k)}$ we have $v_0=w_0,\dots,v_{k-1}=w_{k-1}$ and
$\{v_k,\dots,v_m\}=\{w_k,\dots,w_m\}$ as multisets.

For $t=0,\dots,k-2$ we have $t<k-1$ and also $t<l-1$, hence by \eqref{eq:lujrmatrix}
\begin{align}
    (\kk v)_t=v_0+v_{t+1}=w_0+w_{t+1}=(\ll w)_t.
\end{align}
Thus the first $k-1$ components of $\kk v$ and $\ll w$ coincide.

At the special index $k-1$ we have by \eqref{eq:lujrmatrix}
\begin{align}
(\kk v)_{k-1}=v_0,\qquad (\ll w)_{k-1}=w_0+w_k=v_0+w_k>v_0,    
\end{align}
so $(\kk v)^{(k)}\neq(\ll w)^{(k)}$.

Since $\kk v$ and $\ll w$ have the same multiset of components,
equality of the first $k-1$ entries implies $(\kk v)^{(k-1)}=(\ll w)^{(k-1)}$.
\end{proof}

\subsection{Proof}
First, we explain how Theorem \ref{main_theorem_rephrased} implies Theorem \ref{main_theorem}, before the proof Theorem \ref{main_theorem_rephrased} in the next part.

Since weights determine the small growth vector of the EKR form (and hence determine the strata), it is sufficient to show that sequences of weights associated with two classes, denoted by different lujr codes, are different.
Our goal, then, is to show that two different codes lujr give rise to different sequences. These codes must have a first (from the right) difference at some position, on which they have $k,$ and $l>k$.
We can rephrase it as having two codes  $...kO $ and $...lO $ with the same suffix $O$, give rise to vectors

\begin{align*}
    \textbf{kO}&\begin{bmatrix} 1 \\ ... \\ 1\end{bmatrix}, K_1 \textbf{kO}\begin{bmatrix} 1 \\ ... \\ 1\end{bmatrix},K_2K_1 \textbf{kO}\begin{bmatrix} 1 \\ ... \\ 1\end{bmatrix},\hdots\\
    \textbf{lO}&\begin{bmatrix} 1 \\ ... \\ 1\end{bmatrix}, L_1 \textbf{lO}\begin{bmatrix} 1 \\ ... \\ 1\end{bmatrix},L_2L_1 \textbf{lO}\begin{bmatrix} 1 \\ ... \\ 1\end{bmatrix},\hdots.
\end{align*}
Let us denote the common suffix $O$ acting on $e_0+\hdots+e_{m}$ as 
\begin{align}
    v=\textbf{O}&\begin{bmatrix} 1 \\ ... \\ 1\end{bmatrix}=w.\end{align}
Then $v=w,$ therefore $v^{(k)}=w^{(k)}$ for $k=0,1,\hdots, m+1$. For us, this means that the assumptions of the Theorem \ref{main_theorem_rephrased} are satisfied, so if we prove it, we will conclude with the following corollary.
\begin{crllr}\label{important_corollary}
 EKR normal form corresponding to different codes $\j'_1 \j'_2 \hdots \j'_r\neq \j_1 \j_2 \hdots \j_r$ have different sequences of weights. Here, by sequence of weights, we mean a sequence of $m$ largest components of vectors
 \begin{align}
     K_1v,\ K_2K_1v,\hdots \qquad \text{ and } \qquad  L_1v,\ L_2L_1v,\hdots,
 \end{align}
and vectors $v,w$ can be chosen as vectors obtained from $(e_0+e_1+\hdots+e_m)$ by the action of operators given by a common suffix of codes as well as first characters differentiating these codes.

Small growth vectors are common for all strongly nilpotent germs in a singularity class, and they can be obtained in the EKR normal form from a given class. Therefore singularity classes are inequivalent as the Theorem \ref{main_theorem} states.
\end{crllr}

\begin{dfntn}
For two vectors $v,w\in\mathbb{Z}_{>0}^{m+1}$ define the \emph{prefix-agreement depth}
\begin{align}
    \mathrm{pr}(v,w):=\max\{\,s\in\{0,1,\dots,m\}\;:\; v^{(s)}=w^{(s)}\,\}.
\end{align}
Thus $\mathrm{pr}(v,w)=s$ means $v^{(s)}=w^{(s)}$ but $v^{(s+1)}\neq w^{(s+1)}$ (unless $s=m$).
\end{dfntn}

\begin{proof}[Proof of Theorem~\ref{main_theorem_rephrased}]
Let the two distinct lujr codes have the first (from the right) discrepancy at a pair of letters $k<l$
and share the same suffix $O$, so they are of the form
\[
P\,k\,O \qquad\text{and}\qquad Q\,l\,O,
\]
for (possibly empty) prefixes $P,Q$ on the left.

Set
\[
u:=\mathbf{O}\begin{bmatrix}1\\ \vdots\\ 1\end{bmatrix},\qquad v:=\kk u,\qquad w:=\ll u.
\]
Since $u^{(k)}=u^{(k)}$, Lemma~\ref{equal} (applied with  $k<l$) gives
\[
v^{(k-1)}=w^{(k-1)}\quad\text{and}\quad v^{(k)}\neq w^{(k)},
\]
so $\mathrm{pr}(v,w)=k-1$.

Now extend both words leftward step by step. At any intermediate stage we have vectors $\tilde v,\tilde w$
and next letters $a$ (in the first word) and $b$ (in the second word), giving
\[
\tilde v':=a\tilde v,\qquad \tilde w':=b\tilde w.
\]
Denote $s:=\mathrm{pr}(\tilde v,\tilde w)$, so $\tilde v^{(s)}=\tilde w^{(s)}$ and $\tilde v^{(s+1)}\neq \tilde w^{(s+1)}$.
First, we claim that $\mathrm{pr}$ never increases, and it strictly drops unless we are in the “stable” situation
$a=b\le s$.

\smallskip
\noindent\textbf{(I) $a=b=j\le s$.}

 $j\le s$ so Lemma~\ref{obserwacje}.2 yields $(\j\tilde v)^{(s)}=(\j\tilde w)^{(s)}$ and
$(\j\tilde v)^{(s+1)}\neq (\j\tilde w)^{(s+1)}$, hence $\mathrm{pr}(\tilde v',\tilde w')=s$.

\smallskip
\noindent\textbf{(II) $a=b=j>s$.}

 $j>s$ (and $s>0$) so Lemma~\ref{obserwacje}.1 yields $(\j\tilde v)^{(s-1)}=(\j\tilde w)^{(s-1)}$ and
$(\j\tilde v)^{(s)}\neq (\j\tilde w)^{(s)}$, hence $\mathrm{pr}(\tilde v',\tilde w')=s-1$.

\smallskip
\noindent\textbf{(III) $a\neq b$.}
Let $c:=\min(a,b)$ and $d:=\max(a,b)$.
\begin{itemize}
\item If $c\le s$, then $\tilde v^{(c)}=\tilde w^{(c)}$ (because $c\le s$), so Lemma~\ref{equal} applied to $c<d$ gives
\begin{align}
 (c\tilde v)^{(c-1)}=(d\tilde w)^{(c-1)}\quad\text{and}\quad (c\tilde v)^{(c)}\neq(d\tilde w)^{(c)},   
\end{align}
hence $\mathrm{pr}(\tilde v',\tilde w')=c-1\le s-1$.
\item If $c>s$, then in fact $a>s$ and $b>s$. For indices $0\le r\le s-2$ we have $r<a-1$ and $r<b-1$, so by the explicit formula
for the action of a letter,
\begin{align}
(a\tilde v)_r=\tilde v_0+\tilde v_{r+1}=\tilde w_0+\tilde w_{r+1}=(b\tilde w)_r,    
\end{align}
since $\tilde v_0=\tilde w_0$ and $\tilde v_{r+1}=\tilde w_{r+1}$ for $r+1\le s-1$.
But at $r=s-1$ we still have $s-1<a-1$ and $s-1<b-1$, hence
\begin{align}
 (a\tilde v)_{s-1}=\tilde v_0+\tilde v_s\neq \tilde w_0+\tilde w_s=(b\tilde w)_{s-1},   
\end{align}
because $\tilde v_s\neq \tilde w_s$. Therefore $\mathrm{pr}(\tilde v',\tilde w')=s-1$.
\end{itemize}

\smallskip
So, as we move leftward, $\mathrm{pr}$ is nonincreasing and drops by at least $1$ whenever we are in the case  (II) or in the case (III).

Notice, however, that after the first distinct pair $k,l$ appeared, the index pr is at the value $k-1$. Then, because of the lujr property of the suffix of the second code $Q$, we will reach letter $k$ in the second code:
\begin{align*}
    \mathbf{P}'\j'\tilde{v},\qquad   \mathbf{Q}'\kk \tilde{w}.
\end{align*}
 Because pr is not increasing, $\mathrm{pr}(\tilde v,\tilde w)<k$ at the letter $k$. This means that the case must be either (II) or (III) so we end up with sequences
\begin{align*}
    \mathbf{P}'\tilde{v}',\qquad   \mathbf{Q}' \tilde{w}', \qquad \mathrm{pr}(\tilde v',\tilde w')<k-1,
\end{align*}
and we know that $Q'$ is lujr and $k-1\in Q'$. We can repeat this reasoning to conclude that $\mathrm{pr}$ will reach $0$ before the beginning of the word (which is the letter $1$).

\smallskip
Now let $s_*$ be the first stage at which $\mathrm{pr}$ reaches $0$. At that stage we have $\tilde v^{(0)}=\tilde w^{(0)}$ but
$\tilde v^{(1)}\neq \tilde w^{(1)}$, i.e.\ the two vectors have the same multiset of components but different first component.

Crucially, the stage $s_*$ cannot be the very beginning of the words,
so there is at least one more letter to be applied after stage $s_*$.

Apply the next letter to obtain $\tilde v_+$ and $\tilde w_+$. The multiset of the $m$ largest components of $\tilde v_+$ equals
\[
\{\tilde v_0+\tilde v_1,\dots,\tilde v_0+\tilde v_m\},
\]
according to \eqref{set_of_weights} and similarly for $\tilde w_+$ 
\[
\{\tilde w_0+\tilde w_1,\dots,\tilde w_0+\tilde w_m\}.
\] Because $\tilde v_0\neq \tilde w_0$, and $\tilde v^{(0)}=\tilde w^{(0)}$, these two multisets differ as their sums are different 
\begin{align}
    \sum_{i=1}^m(\tilde v_0+\tilde v_i)=(m-1)\tilde v_0+\sum_{i=0}^m \tilde v_i\neq (m-1)\tilde w_0+\sum_{i=0}^m \tilde v_i=(m-1)\tilde w_0+\sum_{i=0}^m \tilde w_i=\sum_{i=1}^m(\tilde w_0+\tilde w_i)
,
\end{align}
hence the corresponding weights differ. Therefore the sequences of weights produced by the two codes cannot coincide, proving the theorem.
\end{proof}

\section{Proof of Theorem \ref{An_compact_form_of_thm2} and Theorem \ref{thm:sums_of_Goursats}}\label{app:recursion}
\begin{proof}[Proof of Theorem~\ref{An_compact_form_of_thm2}]
 The identity will be proven inductively. First, one can see that for $j=0$, the set of indices is empty, and for $j=1$, it contains only one element $i_1=0$, so formulas \eqref{Aj_formulas}
    \begin{align}
        A_0=1,\qquad A_1=1+(1+k_0)
    \end{align}
    agree with the \eqref{A_recursion_bezG0}-\eqref{A_recursion_bezG1}.

    Next, we for the general $j$, the recursion  \eqref{A_recursion_bezG2} for $j+1$ can be written as
    \begin{align}
    A_{j+1}&=A_{j-1}+A_{j}(1+k_j)\\
    &=1+\sum_{{\substack{0\leq i_1<...<i_p<j-1\\i_1\not\equiv i_2,...,i_{p-1}\not\equiv i_p (\text{mod }2),\\i_p\not\equiv j-1\equiv j+1(\text{mod }2)}}} \prod_{l=1}^p(1+k_{i_l}) + (1+k_j)+(1+k_j)\sum_{{\substack{0\leq i_1<...<i_p<j\\i_1\not\equiv i_2,...,i_{p-1}\not\equiv i_p,i_p\not\equiv j(\text{mod }2)}}}\prod_{l=1}^p(1+k_{i_l})\\
    &1+\sum_{{\substack{0\leq i_1<...<i_p<j-1\\i_1\not\equiv i_2,...,i_{p-1}\not\equiv i_p,i_p\not\equiv j+1(\text{mod }2)}}} \prod_{l=1}^p(1+k_{i_l}) +\sum_{{\substack{0\leq i_1<...<i_p=j\\i_1\not\equiv i_2,...,i_{p-2}\not\equiv i_{p-1},i_{p-1}\not\equiv j=i_p(\text{mod }2)}}}\prod_{l=1}^p(1+k_{i_l}) \label{mid-equality_rec1}\\
    &=1+\sum_{{\substack{0\leq i_1<...<i_p<j+1\\i_1\not\equiv i_2,...,i_{p-1}\not\equiv i_p,i_p\not\equiv j+1(\text{mod }2)}}} \prod_{l=1}^p(1+k_{i_l}),
\end{align}
where the equation between the second and the third line is adding an index $i_{p+1}=j$ is the second sum, also for $p=0$ ($1$- representing an empty product). In the line \eqref{mid-equality_rec1}, the first term represents the sum over indices with $i_p<j-1$ and the second term the ones with $i_p=j$, which together are equal to the expression \eqref{Aj_formulas} for $j+1$.
\end{proof}

\begin{proof}[Proof of Theorem~\ref{thm:sums_of_Goursats}]
First, $n\leq s$:
    \begin{align}S_n&=(2+k_0)A_0+\sum_{j=1}^{n}A_j(1+k_j)\\
    &=1+(1+k_0)+\sum_{j=1}^{n}(1+k_j)\biggl(1+\sum_{{\substack{0\leq i_1<...<i_p<j\\i_1\not\equiv i_2,...,i_{p-1}\not\equiv i_p,i_p\not\equiv j(\text{mod }2)}}} \prod_{l=1}^p(1+k_{i_l}) \biggr)\\
    &=1+\sum_{j=1}^{n}\sum_{{\substack{0\leq i_1<...<i_p=j\\i_1\not\equiv i_2,...,i_{p-1}\not\equiv i_p(\text{mod }2)}}} \prod_{l=1}^p(1+k_{i_l}) \\
    &=1+\sum_{j=0}^{n}\sum_{{\substack{0\leq i_1<...<i_p=j\\i_1\not\equiv i_2,...,i_{p-1}\not\equiv i_p(\text{mod }2)}}} \prod_{l=1}^p(1+k_{i_l}) \\
    &=1+\sum_{{\substack{0\leq i_1<...<i_p\leq n\\i_1\not\equiv i_2,...,i_{p-1}\not\equiv i_p(\text{mod }2)}}} \prod_{l=1}^p(1+k_{i_l}).
\end{align}
Above, we fixed the last index $i_p$ in the multiindex $(i_1,\hdots, i_p)$ to be equal to $j$.

When $n=s+1$, we add $A_{s+1}$ with its multiplicity
\begin{align}
    &S_{s+1}=S_s+ (k_{s+1}-1)A_{s+1} \\ 
    &=1+(k_{s+1}-1)+\sum_{{\substack{0\leq i_1<...<i_p\leq s\\i_1\not\equiv i_2,...,i_{p-1}\not\equiv i_p(\text{mod }2)}}} \prod_{l=1}^p(1+k_{i_l}) +(k_{s+1}-1)\sum_{{\substack{0\leq i_1<...<i_p<s+1\\i_1\not\equiv i_2,...,i_{p-1}\not\equiv i_p,i_p\not\equiv s+1(\text{mod }2)}}}  \prod_{l=1}^p(1+k_{i_l})\\
    &=k_{s+1} +k_{s+1}\cdot \sum_{{\substack{0\leq i_1<...<i_p<s+1\\i_1\not\equiv i_2,...,i_{p-1}\not\equiv i_p,i_p\not\equiv s+1(\text{mod }2)}}}   \prod_{l=1}^p(1+k_{i_l}) + \sum_{{\substack{0\leq i_1<...<i_p<s+1\\i_1\not\equiv i_2,...,i_{p-1}\not\equiv i_p,i_p\equiv s+1(\text{mod }2)}}}  \prod_{l=1}^p(1+k_{i_l}).
\end{align}
\end{proof}

\vspace{2em}
\noindent{\bf Acknowledgments:}
I am deeply grateful to Piotr Mormul for his guidance, help and continued support, and in particular for a key simplification of the proof of Theorem~\ref{main_theorem_rephrased}. Special thanks go to Michał Jóźwikowski for numerous helpful suggestions.

\bibliographystyle{plain}
\bibliography{literature}

@InProceedings{Mormul_DoModuli,
author = {Mormul, P.},
editor={Brasselet, J. P.
and Ruas, M. A. S.},
year = {2007},
month = {01},
pages = {229-246},
title = {Do Moduli of Goursat Distributions Appear on the Level of Nilpotent Approximations?},
booktitle={Real and Complex Singularities},
isbn = {978-3-7643-7775-5},
doi = {10.1007/978-3-7643-7776-2_17}
}

@incollection{Lafferriere1993,
author="Lafferriere, G.
and Sussmann, H. J.",
editor="Li, Zexiang
and Canny, J. F.",
title="A Differential Geometric Approach to Motion Planning",
bookTitle="Nonholonomic Motion Planning",
year="1993",
publisher="Springer US",
address="Boston, MA",
pages="235--270",
isbn="978-1-4615-3176-0",
doi="10.1007/978-1-4615-3176-0_7",
url="https://doi.org/10.1007/978-1-4615-3176-0_7"
}

@article{Mormul_2023_multi, title={Singularity classes of special multi-flags, {I}}, volume={16}, DOI={10.15673/pigc.v16i2.2336}, number={2}, journal={Proceedings of the International Geometry Center}, author={Mormul, P.}, year={2023}, pages={142–160}}

@article{Mormul_2004, title={Multi-dimensional {C}artan prolongation and special K–flags}, volume={65}, DOI={10.4064/bc65-0-12}, number={1}, journal={Banach Center Publications}, author={Mormul, P.}, year={2004}, pages={157–178}}

@article{SGV_Mor_2004, title={Geometric classes of {G}oursat flags and the arithmetics of their encoding by small growth vectors}, volume={2}, DOI={10.2478/bf02475982}, number={5}, journal={Central European Journal of Mathematics}, author={Mormul, P.}, year={2004}, pages={859–883}}

@article{Mormul2011SmallGV,
  title={Small growth vectors of the compactifications of the contact systems on {$J^r(1,1)$}},
  author={Mormul, P.},
  journal={Contemporary Mathematics},
  volume = {569},
  year={2012},
  pages={123-141},
}

@article{Jean_1996,
     author = {Jean, F.}, 
     title = {The car with $N$ trailers: characterisation of the singular configurations},
     journal = {ESAIM: Control, Optimisation and Calculus of Variations},
     pages = {241--266},
     publisher = {SMAI (Soci\'et\'e de math\'ematiques appliqu\'ees et industrielles)},
     address = {Paris},
     volume = {1},
     year = {1996},
     zbl = {0874.93033},
     mrnumber = {1411581},
     language = {en},
     url = {http://www.numdam.org/item/COCV_1996__1__241_0/}
}

@article{Zhitomirskii1991NormalFO,
  title={Normal forms of germs of smooth distributions},
  author={M. Zhitomirskii},
  journal={Mathematical notes of the Academy of Sciences of the USSR},
  year={1991},
  volume={49},
  pages={139-144}
}

@article{AgrachevMarigo03,
  title={Nonholonomic tangent spaces: intrinsic construction and rigid dimensions},
  author={Agrachev, A. and Marigo, A.},
  journal={Electronic Research Announcements of The American Mathematical Society},
  year={2003},
  volume={9},
  pages={111-120},
  url={https://api.semanticscholar.org/CorpusID:51820248}
}

@article{Mormul2005KRalgebrasoptimal,
author = {Mormul, P.},
year = {2005},
month = {04},
pages = {1614-1629},
title = {{K}umpera-{R}uiz algebras in {G}oursat flags are optimal in small lengths},
volume = {126},
journal = {Journal of Mathematical Sciences},
doi = {10.1007/s10958-005-0051-0}
}

@article{giaro1978lectureKR,
  title={Sur la lecture correcte d’un r{\'e}sultat d’Elie Cartan},
  author={Giaro, A. and Kumpera, A. and Ruiz, C.},
  journal={CR Acad. Sci. Paris},
  volume={287},
  pages={241--244},
  year={1978}
}

@incollection{Bellaiche96,
author="Bella{\"i}che, A.",
editor="Bella{\"i}che, Andr{\'e}
and Risler, Jean-Jacques",
title="The tangent space in sub-{R}iemannian geometry",
bookTitle="Sub-Riemannian Geometry",
year="1996",
publisher="Birkh{\"a}user Basel",
address="Basel",
pages="1--78",
isbn="978-3-0348-9210-0",
doi="10.1007/978-3-0348-9210-0_1"
}

@article{Zhitomirskii1995SingularitiesAN,
  title={Singularities and normal forms of smooth distributions},
  author={M. Zhitomirskii},
  journal={Banach Center Publications},
  year={1995},
  volume={32},
  pages={395-409}
}

@article{mormul2024localmodulispecial2flags,
    title = {Local moduli in the special 2-flags of length 5},
    journal = {Differential Geometry and its Applications},
    volume = {104},
    pages = {102402},
    year = {2026},
    issn = {0926-2245},
    doi = {https://doi.org/10.1016/j.difgeo.2026.102402},
    url = {https://www.sciencedirect.com/science/article/pii/S0926224526000732},
    author = {Mormul, P.}
}

@InProceedings{MomrulGoursatdim7,
author={Mormul, P.},
editor="Zinober, A.
and Owens, D.",
title="{G}oursat distributions not strongly nilpotent in dimensions not exceeding seven",
booktitle="Nonlinear and Adaptive Control",
year="2003",
publisher="Springer Berlin Heidelberg",
address="Berlin, Heidelberg",
pages="249--261",
isbn="978-3-540-45802-9"
}

@article{Mormul_2009Singularity2flags,
   title={Singularity Classes of Special 2-Flags},
   ISSN={1815-0659},
   url={http://dx.doi.org/10.3842/SIGMA.2009.102},
   DOI={10.3842/sigma.2009.102},
   journal={Symmetry, Integrability and Geometry: Methods and Applications},
   publisher={SIGMA (Symmetry, Integrability and Geometry: Methods and Application)},
   author={Mormul, P.},
   year={2009}}

@article{Mormulnotexceedingsix,
author = {Mormul, P.},
journal = {Zeszyty Naukowe Uniwersytetu Jagiellońskiego. Universitatis Iagellonicae Acta Mathematica},
language = {eng},
pages = {15-29},
publisher = {Wydawnictwo Uniwersytetu Jagiellońskiego},
title = {Minimal nilpotent bases for {G}oursat distributions of coranks not exceeding six.},
url = {http://eudml.org/doc/127468},
volume = {1278(42)},
year = {2004},
}

@article{Agrachev2001ONTS,
  title={ON THE SUBANALYTICITY OF {C}ARNOT-{C}ARATHEODORY DISTANCES},
  author={A. A. Agrachev and J. P. Gauthier},
  journal={Annales De L Institut Henri Poincare-analyse Non Lineaire},
  year={2001},
  volume={18},
  pages={359-382},
  url={https://api.semanticscholar.org/CorpusID:28989925}
}

@article{engel1890,
  title={Zur Invariantentheorie der Systeme von {P}faffschen Gleichungen},
  author={Engel, F.},
  journal={Berichte Ges. Leipzig, Math-Phys. Classe},
  year={1889},
  volume={XLI},
  pages={157–176}
}

@article{vonWeber,
  title={Zur Invariantentheorie der Systeme {P}faff’scher Gleichungen},
  author={E. von Weber},
  journal={Berichte Ges. Leipzig, Math-Phys. Classe},
  year={1898},
  volume={L},
  pages={207–229},
}

@article{CartanSurLA,
  title={Sur l'{\'e}quivalence absolue de certains syst{\`e}mes d'{\'e}quations diff{\'e}rentielles et sur certaines familles de courbes},
  author={E. Cartan},
  journal={Bulletin de la Soci{\'e}t{\'e} Math{\'e}matique de France},
  year={1914},
  volume={42},
  pages={12-48},
  url={https://api.semanticscholar.org/CorpusID:124775040}
}

@article{KumperasurlequivalencedePfaff,
author = {Kumpera, A. and Ruiz, C.},
year = {1982},
month = {01},
pages = {},
title = {Sur l’équivalence locale des systèmes de {P}faff en {D}rapeau},
journal = {Monge-Ampère Equations and Related Topics}
}

@article{Kumpera2002MultiflagSA,
  title={Multi-flag systems and ordinary differential equations},
  author={A. Kumpera and J. L. Rubin},
  journal={Nagoya Mathematical Journal},
  year={2002},
  volume={166},
  pages={1 - 27}
}

@book{Goursat1923LeonsSL,
  title={Leçons sur le probleme de {P}faff},
  author={E. Goursat},
  publisher={Paris Libraire Scientifuque },
  year={1922},
  address ={Paris} ,
editor={J. Hermann}
}

@article{Montgomery2001GeometricAT,
  title={Geometric approach to {G}oursat flags},
  author={R. Montgomery and M. Zhitomirskii},
  journal={Annales De L Institut Henri Poincare-analyse Non Lineaire},
  year={2001},
  volume={18},
  pages={459-493},
  url={https://api.semanticscholar.org/CorpusID:16669128}
}

@article{Bryant1993RigidityOI,
  title={Rigidity of integral curves of rank 2 distributions},
  author={R. L. Bryant and L. Hsu},
  journal={Inventiones mathematicae},
  year={1993},
  volume={114},
  pages={435-461}
}

@article{Hermes1984NilpotentBF,
  title={Nilpotent bases for distributions and control systems},
  author={H.  Hermes and A.  T. Lundell and D. P. Sullivan},
  journal={Journal of Differential Equations},
  year={1984},
  volume={55},
  pages={385-400},
  url={https://api.semanticscholar.org/CorpusID:54174072}
}

@article{Mormul_Pelleter_2flags,
author = {Mormul, P. and Pelletier, F.},
year = {2010},
month = {11},
pages = {},
title = {Special 2-flags in lengths not exceeding four: a study in strong nilpotency of distributions}
}

@article{Mormul2022MonsterTF,
  title={Monster towers from differential and algebraic viewpoints},
  author={Mormul, P.},
  journal={Journal of Singularities},
  year={2022},
volume={25},
pages = {331-347}
}

@article{Mormul2004MultidimensionalCP,
  title={Multi-dimensional {C}artan prolongation and special k-flags},
  author={Mormul, P.},
  journal={Banach Center Publications},
  year={2004},
  volume={65},
  pages={157-178},
}

@article{CASTRO12,
title = {A Monster Tower approach to {G}oursat multi-flags},
journal = {Differential Geometry and its Applications},
volume = {30},
number = {5},
pages = {405-427},
year = {2012},
issn = {0926-2245},
doi = {https://doi.org/10.1016/j.difgeo.2012.06.005},
author = {A. Castro and Wyatt, H.},
}

@article{CASTRO17_1,
author = {Castro, A. and Howard, W. and Shanbrom, C.},
year = {2017},
month = {06},
pages = {317-333},
title = {Complete spelling rules for the Monster tower over three-space},
volume = {9},
journal = {Journal of Geometric Mechanics},
doi = {10.3934/jgm.2017013}
}

@article{CASTRO17_2,
author = {Castro, A. and Colley, S. and Kennedy, G. and Shanbrom, C.},
year = {2016},
month = {06},
pages = {},
title = {A Coarse Stratification of the Monster Tower},
volume = {66},
journal = {The Michigan Mathematical Journal},
doi = {10.1307/mmj/1508896892}
}

@article{SHIBUYA2009793,
title = {Drapeau theorem for differential systems},
journal = {Differential Geometry and its Applications},
volume = {27},
number = {6},
pages = {793-808},
year = {2009},
issn = {0926-2245},
doi = {https://doi.org/10.1016/j.difgeo.2009.03.017},
author = {K. Shibuya and K. Yamaguchi}
}

@article{Adachi_Jiro,
author = {Adachi, J.},
year = {2010},
month = {12},
pages = {29-56},
title = {Global stability of special multi-flags},
volume = {179},
journal = {Israel Journal of Mathematics},
doi = {10.1007/s11856-010-0072-3}
}

@book{Montgomery2002,
  author    = {Montgomery, R.},
  title     = {A Tour of Subriemannian Geometries, Their Geodesics and Applications},
  year      = {2002},
  publisher = {American Mathematical Society},
  series    = {Mathematical Surveys and Monographs},
  volume    = {91},
  address   = {Providence, RI},
  isbn      = {978-0-8218-4165-5},
  doi       = {10.1090/surv/091}
}

@article{Sussmann1987,
author = {Sussmann, H. J.},
title = {A General Theorem on Local Controllability},
journal = {SIAM Journal on Control and Optimization},
volume = {25},
number = {1},
pages = {158-194},
year = {1987},
doi = {10.1137/0325011},
URL = { https://doi.org/10.1137/0325011},
eprint = {   https://doi.org/10.1137/0325011}
}

@article{Sussmann1983,
  author  = {Sussmann, Hector J.},
  title   = {Lie Brackets and Local Controllability: A Sufficient Condition for Scalar-Input Systems},
  journal = {SIAM Journal on Control and Optimization},
  volume  = {21},
  number  = {5},
  pages   = {686--713},
  year    = {1983},
  doi     = {10.1137/0321042}
}

\end{document}